\documentclass{amsart}
\usepackage{mymathstyle}
\newcommand\blfootnote[1]{%
  \begingroup
  \renewcommand\thefootnote{}\footnote{#1}%
  \addtocounter{footnote}{-1}%
  \endgroup
}

\title{Frobenius-Perron Theory of Modified ADE Bound Quiver Algebras}
\author{Elizabeth Wicks}

\address{Wicks: Department of Mathematics, Box 354350,
University of Washington, Seattle, Washington 98195, USA}
\email{lizwicks@uw.edu}

\begin{document}

\maketitle

\begin{abstract}
	The Frobenius-Perron dimension for an abelian category was recently introduced in \cite{CGWZZZ}. We apply this theory to the category
of representations of the finite-dimensional radical square zero algebras associated to certain modified ADE graphs. In particular, we take an ADE quiver with arrows in a certain orientation and an arbitrary number of loops at each vertex. We show that the Frobenius-Perron dimension of this category is equal to the maximum number of loops at a vertex. Along the way, we introduce a result which can be applied in general to calculate the Frobenius-Perron dimension of a radical square zero bound quiver algebra. We use this result to introduce a family of abelian categories which produce arbitrarily large irrational Frobenius-Perron dimensions. 
\end{abstract}



\vsp

\section{Introduction}

\blfootnote{2010 Mathematics Subject Classification: Primary 16B50, 16G20.}

The Frobenius-Perron dimension of a finite tensor category has proven to be an important invariant since it was introduced by Etingof-Nikshych-Ostrik \cite{ENO}. One can define the Frobenius-Perron dimension of an object in a tensor category as well as the dimension of the tensor category itself, and these dimensions encode many useful properties. For example, the Frobenius-Perron dimension of a finite tensor category is invariant under Morita equivalence \cite{EO}. Furthermore, one can determine whether a finite tensor category is equivalent to the representation category of a finite-dimensional quasi-Hopf algebra by examining the Frobenius-Perron dimension of the objects \cite{EO}. These ideas have been essential to the classification of fusion categories and understanding the representation theory of semi-simple Hopf algebras. 

In 2017, a new type of Frobenius-Perron dimension was introduced, which we will denote $\fpd$ \cite{CGWZZZ}. The authors extended the definition of Frobenius-Perron dimension of an object in a finite tensor category to a much more general setting \cite[Ex. 2.11]{CGWZZZ}. In fact, this definition of $\fpd$ can be applied to any $\kk$-linear category along with a chosen endofunctor. The authors focus on understanding the $\fpd$ of abelian categories (see \cref{xxdef2.1}) and derived categories, where the endofunctor is the suspension. The results in \cite{CGWZZZ} suggest that the Frobenius-Perron dimension of an abelian or derived category is an important and useful invariant with connections to representation theory. The authors hope to continue to explore the properties of the $\fpd$ in the derived and abelian settings and to develop more applications. 

The intent of this paper is to gain a better understanding of the $\fpd$ as it relates to the representation theory of finite-dimensional algebras. First, we review an important theorem from \cite{CGWZZZ}, which is analogous to the finite tensor category result \cite[Prop. 6.3.3]{EGNO}. 

\begin{theorem} \label{thm:CGWZZZ} \cite{CGWZZZ} If $F: \Acal \rightarrow \Bcal$ is an embedding of abelian categories, then $\fpd \Acal \leq \fpd \Bcal.$
\end{theorem}

The previous theorem is useful in that we can use the $\fpd$ to show non-existence of embeddings between abelian categories. However, calculating the $\fpd$ can be a difficult task. If we could develop strategies for calculating the $\fpd$ in certain abelian categories, we could compare those categories via their $\fpd$. To that end, we prove the following theorem, which greatly simplifies the calculation of the $\fpd$ for radical square zero finite-dimensional algebras. First we need some hypotheses.

\begin{hypothesis} \label{hyp:quiver} 
Let $A=\kk Q/(\geq 2)$ for some finite quiver $Q$, where $(\geq 2)$ is the ideal generated by paths of length greater than 1, and let $J$ be the ideal of $A$ generated by the loops. Let $Q'$ be the quiver formed from $Q$ by removing all the loops. Define $C:=\kk Q' \cong \kk Q/\tilde{J}$, where $\tilde{J}$ is the ideal of $\kk Q$ generated by the loops, and $B:=A/J \cong \kk Q'/(\geq 2)=C/(\geq 2)$. Let $\tilde{P}_i$ (respectively $\tilde{I}_i, \tilde{S}_i$) denote the indecomposable projective (respectively injective, simple) $B$-module at each vertex $i$. Note that as $A$-modules, $\tilde{S}_i \cong S_i$ for each $i$.
\end{hypothesis}

Calculating the $\fpd$ requires an understanding of the brick modules (\cref{xxdef2.1}).

\begin{theorem} \label{lem:quivpowerful}
	Assume \cref{hyp:quiver}. 
	Then there is a one-to-one correspondence between the isomorphism classes below:
	$$\{ \text{brick } A\dash modules \} \longleftrightarrow \{ \text{brick } C \dash \text{modules annihilated by } (\geq 2)\}$$
		Furthermore, if $M,N$ are brick $A$-modules, we have a natural isomorphism
	\[
	\Hom_A(M,N) \cong \Hom_B(M,N) \cong \Hom_{C}(M,N).
	\]
\end{theorem}

We apply this theorem to some examples. We hope that these results in conjuction with \cref{thm:CGWZZZ} will be of use to those who are trying to compare or understand these categories.

In particular, we focus on a family of bound quiver algebras related to the $ADE$ quiver algebras (\cref{def:modifiedADE}). The ADE quivers and related quivers mysteriously appear in settings as diverse as the classification of semisimple Lie algebras \cite{Humphreys}, the McKay correspondence \cite{McKay}, the classification of minimal models of two-dimensional conformal field theory \cite{Cappelli}, and more. In addition to their ubiquity, the $ADE$ quivers are of finite representation type and their representations are well-known \cite{Gabriel}, making variations of their representation categories ideal to study. 

\begin{theorem} \label{thm:ADE}
Let $A$ be a modified ADE bound quiver algebra with $N_i$ loops at each vertex $i$ (\cref{def:modifiedADE}). Then
\[
\fpd \l(A \dash \bmod\r) = \max\{N_1, \ldots, N_n\}.
\]
\end{theorem}

Some consequences of this theorem include non-existence of embeddings. Suppose that $A$ is a modified $ADE$ bound quiver algebra with at least one loop on a single vertex. Then the category of representations of $A$ does not embed into any of the categories of representations of the $A_n, D_n,$ or $E_n$ quiver algebras, since these have $\fpd$ 0 \cite{CGWZZZ}. Similarly, if $A$ is a modified $ADE$ bound quiver algebra with at least two loops on a single vertex, the category of representations of $A$ does not embed into any of the categories of representations of the $\tilde{A}_n,\tilde{D}_n,$ or $\tilde{E}_n$ quiver algebras, since these have $\fpd$ 1 \cite{CGWZZZ}. And if $A,B$ are modified $ADE$ bound quiver algebras where the maximum number of loops at a single vertex of $A$ is $N$, and the maximum number of loops of at a single vertex of $B$ is $M$, then if $N < M$, $B\dash\bmod$ cannot embed into $A\dash\bmod$.

The following question would be interesting to investigate.

\begin{question} What is the $\fpd$ for the $\tilde{A}\tilde{D}\tilde{E}$ modified bound quiver algebras?
\end{question}

Notice that \cref{ex:Qnm} is of type $\tilde{A}(1)$. In this case we have a formula for the $\fpd$ which depends on the number of loops at each vertex, which is irrational in many cases. So we do not get the same result as \cref{thm:ADE}. It would be interesting to see if there is a formula that describes the $\fpd$ of the $\tilde{A}\tilde{D}\tilde{E}$ modified bound quiver algebras.

It is natural to ask about the $\fpd$ of the derived category.

\begin{question}
Let $A$ be a modified $ADE$ bound quiver algebra. What is $\fpd D^b(A\dash\bmod)?$
\end{question}

We defined the modified ADE bound quiver algebras with a certain arrow orientation (\cref{def:modifiedADE}). We ask if the $\fpd$ is invariant under changing directions of arrows, and show that the question has a positive answer up to $n=3$ (\cref{thm:ADEtilde3}).

\begin{question} \label{conj:ADE}
Let $Q'$ be a quiver whose underlying graph is $A_n, D_n,$ or $E_n$ for some $n$. Let $Q$ be the quiver formed from $Q'$ by adding $N_i$ loops to each vertex $i$, and let $A=kQ/(\geq 2)$. Is it true that
	\[
	\fpd \l(A \dash \bmod\r) = \max\{N_1, \ldots, N_n\}?
	\]
\end{question}

The $\fpd$ of an abelian category is an element of $\R^{\geq 0} \cup \{\pm \infty\}$. It is natural to ask

\begin{question}
For a finite-dimensional algebra $A$, what values of $\fpd A\dash \bmod$ are possible?
\end{question}

As a partial answer to this question, we provide the following family of examples using \cref{lem:quivpowerful}. The general answer to this question is unknown.

\begin{proposition} \label{prp:Qnm}
There exists a family of finite-dimensional radical square zero algebras $A$ which can be used to construct arbitrarily large irrational values of $\fpd \l(A\dash\bmod\r)$.
\end{proposition}

\subsection{Conventions and definitions}

We assume that the base field $\kk$ is algebraically closed. 

If $A$ is a finite-dimensional $\kk$-algebra, then $A\dash\bmod$ (respectively $\bmod\dash A$) denotes the category of finite-dimensional left (respectively right) $A$-modules. We will use $A^{op}$ to denote the opposite algebra of $A$.

Let $Q$ denote a finite quiver. The ideal generated by all paths of lengths greater than $1$ in $\kk Q$ is denoted by $(\geq 2)$. If $A=\kk Q/I$ is a bound quiver algebra, we denote the indecomposable projective (respectively injective, simple) module at vertex $i$ by $P_i$ (respectively $I_i$, $S_i$).

The spectral radius of a square matrix $M$ with entries in $\R$ or $\C$ is denoted $\rho(M)$ \cite[Defn. 1.2.9]{Horn}.

We review the relevant definitions, which can be found in \cite{CGWZZZ}.   

\begin{definition}
\label{xxdef2.1} 
Let $\Ccal$ be a $\kk$-linear abelian category, and let $\phi:=\{X_1, X_2, \ldots,X_n\}$ be a finite subset of nonzero
objects in ${\mathcal C}$.
\begin{enumerate}
\item[(1)]
The {\it adjacency matrix} of $\phi$ is defined to be
$$A(\phi):=(a_{ij})_{n\times n}, \quad 
{\text{where}}\;\; a_{ij}:=\dim \Ext^1_{\Ccal}(X_i, X_j) \;\;\forall i,j.$$
\item[(2)]
An object $M$ in ${\mathcal C}$ is called a {\it brick} 
\cite[Definition 2.4, Ch. VII]{ASS} if 
\begin{equation}
\notag
\Hom_{\mathcal C}(M,M)=\kk. 
\end{equation}
\item[(3)]
$\phi\in \Phi$ is called a {\it brick set} if each $X_i$ is a brick and 
$$\dim \Hom_{\Ccal}(X_i, X_j)=\delta_{ij}$$
for all $1\leq i,j\leq n$. The set of brick $n$-object subsets is denoted by $\Phi_{n,b}$. We write $\Phi_{b}=\bigcup_{n\geq 1} 
\Phi_{n,b}$. 
\end{enumerate}
\end{definition}

\begin{definition}
\label{xxdef2.3}
The {\it $n$th Frobenius-Perron dimension} of $\Ccal$ is defined to be
$$\fpd^n (\Ccal):=\sup_{\phi\in \Phi_{n,b}}\{\rho(A(\phi))\}.$$
If $\Phi_{n,b}$ is empty, then by convention, $\fpd^n(\Ccal)=-\infty$.

The {\it Frobenius-Perron dimension} of $\Ccal$ is defined to be
$$\fpd (\Ccal):=\sup_n \{\fpd^n(\Ccal)\}
=\sup_{\phi\in \Phi_{b}} \{\rho(A(\phi)) \}.$$
\end{definition}

\begin{definition} \label{def:modifiedADE}

We define a family of algebras closely related to the ADE path algebras.

Define $A(n, N_1, \ldots, N_n)=\kk Q_A/(\geq 2)$ for $n \geq 1$, where $Q_A$ is the following quiver with $N_i$ loops (labeled $a_i^l$ for $l=1, \ldots, N_i$) at vertex $i$:
	
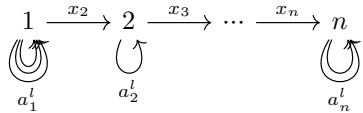
	\[
	\begin{tikzcd}
		1  \arrow[out=252, in=-72, loop] \arrow[out=245, in=-65, loop] \arrow[out=238, in=-58, loop, swap, "a_1^l"] \arrow[r, "x_2"] & 2 \arrow[out=245, in=-65, loop, swap, "a_2^l"] \arrow[r, "x_3"] &  \cdots \arrow[r, "x_{n}"] & n \arrow[out=245, in=-65, loop] \arrow[out=238, in=-58, loop, swap, "a_n^l"]
	\end{tikzcd} 
	\]	
	
Define $D(n,M_1, \ldots, M_n)=\kk Q_D/(\geq 2)$ for $n \geq 4$, where $Q_D$ is the following quiver with $M_i$ loops (labeled $b_i^l$ for $l=1, \ldots, M_i$) at vertex $i$:

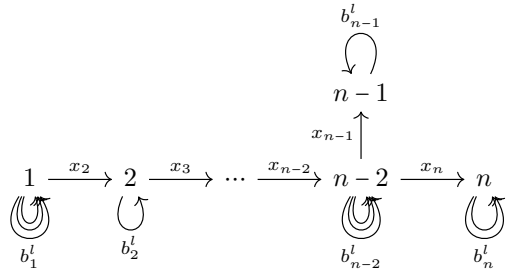
	\[
	\begin{tikzcd}
	{} & {} & {} & n-1 \arrow[out=65, in=115, loop, swap, "b_{n-1}^l"] & {}  \\
	1   \arrow[r, "x_2"]  \arrow[out=252, in=-72, loop] \arrow[out=245, in=-65, loop] \arrow[out=238, in=-58, loop, swap, "b_1^l"] & 2 \arrow[out=245, in=-65, loop, swap, "b_2^l"] \arrow[r, "x_3"] &  \cdots \arrow[r, "x_{n-2}"] & n-2 \arrow[out=252, in=-72, loop] \arrow[out=245, in=-65, loop] \arrow[out=238, in=-58, loop, swap, "b_{n-2}^l"] \arrow[r, "x_n"] \arrow[u, "x_{n-1}"] & n \arrow[out=245, in=-65, loop] \arrow[out=238, in=-58, loop, swap, "b_n^l"]
	\end{tikzcd} 
	\]

Define $E(n,K_1, \ldots, K_n):=\kk Q_E/(\geq 2)$ for $n \in \{6,7,8\}$, where $Q_E$ is the following quiver with $K_i$ loops (labeled $c_i^l$ for $l=1, \ldots, K_i$) at vertex $i$:

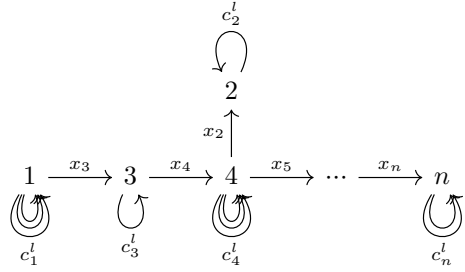
	\[
	\begin{tikzcd}
	{} & {} & 2 \arrow[out=65, in=115, loop, swap, "c_{2}^l"] & {}  \\
	1   \arrow[r, "x_3"]  \arrow[out=252, in=-72, loop] \arrow[out=245, in=-65, loop] \arrow[out=238, in=-58, loop, swap, "c_1^l"] & 3 \arrow[out=245, in=-65, loop, swap, "c_3^l"] \arrow[r, "x_4"] & 4 \arrow[out=252, in=-72, loop] \arrow[out=245, in=-65, loop] \arrow[out=238, in=-58, loop, swap, "c_{4}^l"] \arrow[r, "x_5"] \arrow[u, "x_{2}"] & \cdots \arrow[r, "x_n"] & n  \arrow[out=245, in=-65, loop] \arrow[out=238, in=-58, loop, swap, "c_n^l"]
	\end{tikzcd} 
	\]

We refer to the algebras above as the modified ADE bound quiver algebras. When the number of each loops at each vertex is given, we refer to these algebras as $A(n),D(n)$, and $E(n)$. These algebras appear in \cref{thm:A_n1dir,thm:Dn1dir,thm:En1dir}.
\end{definition}
		

\vsp		
		
\section{Preliminaries}

	We compile a list of results which we use later in the paper. The reader is welcome to skip this section and come back to it for reference. 
	
	\begin{definition} We define the path algebra $\kk Q$ of a finite quiver $Q$ to be the associative $\kk$-algebra determined by the generators $e_i, \alpha$ for all vertices $i$ of $Q$ and all edges $\alpha$ of $Q$, such that for all $i,j,\alpha$,
	\[
	e_ie_j=\delta_{ij}, \quad e_i \alpha=\delta_{i,t(\alpha)}\alpha, \quad \alpha e_j = \delta_{j,s(\alpha)} \alpha.
	\]
	Note that the definiton of path algebra in \cite{ASS} is the opposite algebra of $\kk Q$.
	\end{definition}

	\begin{fact} \label{fct:collection}
		Let $A$ be a finite-dimensional $\kk$-algebra.
		\begin{enumerate}
		\item \cite[Prob. 5.6]{Sch} The following functors are mutually inverse contravariant equivalences
		\begin{align*}
		D(-)&:= \Hom_{\kk}(-,\kk): \bmod\dash A \rightarrow A \dash \bmod, \\
		D(-)&:= \Hom_{\kk}(-,\kk): A \dash \bmod \rightarrow \bmod\dash A.
		\end{align*} 
		
		\item \label{itm:3} \cite[Prob. 5.6]{Sch} $D$ gives duality between injective left modules and projective right modules (and vice versa). This is a consequence of $\Ext^1_A(DX,DY) \cong \Ext^1_A(Y,X).$
		
		\item \label{fct:injinddual} \cite[Prob. 5.6, Prop. 5.7]{Sch} Define the Nakayama functor $\nu(-) = D \circ \Hom_A(-,A).$ Now suppose that $A=\kk Q/		I$ is a bound quiver algebra. We define $I_i:=\nu(Ae_i) = D(e_iA).$ Then $I_i$ is an indecomposable injective left $A$-module.
		
		\item  \label{prp:dimhoms1} \cite[Lem. 1.7.5,1.7.7]{Benson}
		Let $A=\kk Q/I$ be a bound quiver algebra. Then 
		$$\dim \Hom_A(P_i,S_j)= \delta_{ij}=\dim\Hom_A(S_i, I_j).$$
		
		\item \label{prp:dimhoms2} \cite[Prob. 2.8, Cor. 2.12]{Sch}
		Let $A=\kk Q$ be a path algebra. Then
		\[
		\dim \Hom_A(I_j,I_i)= \text{ the number of paths from }i\rightarrow j = \dim \Hom_A(P_j,P_i).
		\]
		
		\item \label{itm:dimextsimples} \cite[Lem. III.2.12(b)]{ASS}
		Let $A=\kk Q/I$ be a bound quiver algebra. Then
		\begin{align*}
		\dim \Ext^1_A(S_i,S_j)= \text{ the number of arrows from } i \rightarrow j.
		\end{align*}
		
		\item \label{itm:8} \cite[Defn. II.1.2, Thm. III.1.6]{ASS} Let $A= \kk Q/I$ be a bound quiver algebra. There exists an equivalence of categories $A \dash \bmod \cong \Rep_{\kk}(Q,I).$ We will freely identify these categories throughout the rest of the paper.
		
		\item \label{prp:equiv} \cite[4.3, Ex. 4.16]{Fuller}
		Let $A$ be a $\kk $-algebra and let $I$ be an ideal. Then there is an equivalence of categories between $\Ccal=A/I\dash\bmod$ and $\Dcal,$ which is defined to be the full subcategory of $A\dash\bmod$ whose objects are those $A$-modules which are annihilated by $I$.
		
		\item \label{prp:notannihgeqn} 
		Let $Q$ be a finite quiver and let $V=(V_i, f_{\alpha}) \in \Rep_{\kk }(Q)$. Suppose $\alpha_1 \ldots \alpha_n$ is a non-zero path, where each $\alpha_i$ is an edge of $Q$. Let $I$ be an ideal of $\kk Q$ containing $\alpha_1 \ldots \alpha_n$. If $f_{\alpha_1} \circ \cdots \circ f_{\alpha_n}$ is non-zero, then $V$ is not annihilated by $I$. This is a consequence of \cref{itm:8}.
			
		\item \label{prp:indreps} \cite[Cor. 1.7.3]{Drozd}
		Suppose that $Q$ is a finite quiver with an indecomposable representation $V=(V_i, f_{\beta}) \in \Rep_{\kk }(Q)$. If $\alpha$ is an edge of $Q$ such that $Q - \alpha$ is disconnected, and $V_{t(\alpha)} \not =0 \not = V_{s(\alpha)},$ then $f_{\alpha}\not =0$.
		
		\item \label{prp:projinjrepQ} \cite[Lem. 2.4, 2.6]{ASS}
		Let $A=\kk Q/(\geq 2)$ be a radical square zero algebra described by a finite quiver $Q$ such that $e_jAe_i$ is one-dimensional for all $j$. Then $P_i=Ae_i$ corresponds to the representation $(V_j, f_{\alpha})$, where 
		\begin{align*}
		V_j&=\begin{cases}
			&\kk \text{ if } j=i \\
			&\kk \text{ if } j=t(\alpha) \text{ for some arrow }\alpha \text{ with source }i	 \\
			& 0 \text{ otherwise}	
			\end{cases}, \\
		f_{\alpha} &=	\begin{cases}
					&1 \text{ if } V_{s(\alpha)}=\kk =V_{t(\alpha)} \\
					& 0 \text{ otherwise}
					\end{cases}.
		\end{align*}
		
		Similarly, $I_i=D(e_iA)$ corresponds to the representation $(W_j, g_{\beta})$, where 
		\begin{align*}
		W_j&=\begin{cases}
			&\kk  \text{ if } j=i \\
			&\kk  \text{ if } j=s(\beta)\text{ for some arrow }\beta \text{ with target }i	 \\
			& 0 \text{ otherwise}	
			\end{cases}, \\
		g_{\beta} &=	\begin{cases}
					&1 \text{ if } W_{s(\beta)}=\kk =W_{t(\beta)} \\
					& 0 \text{ otherwise}
					\end{cases}.
		\end{align*}
		\end{enumerate}
		\end{fact}

We need the following results for our calculations. 

	\begin{lemma} \label{lem:rad}
	Assume \cref{hyp:quiver}. 
	
	\begin{enumerate}
	\item \label{lem:radproj} If $\tilde{P}_i, \tilde{P}_j$ are in the same brick set in $A\dash\bmod$, then we have $\Ext^1_A(\tilde{P}_i,\tilde{P}_j)=0$ if $i\not = j$. If $\tilde{P}_i \not \cong S_i$, then $\Ext^1_A(\tilde{P}_i,\tilde{P}_i)=0.$
	
	\item \label{lem:radinj} Dually, if $\tilde{I}_i, \tilde{I}_j$ are in the same brick set, then we have $\Ext^1_A(\tilde{I}_i,\tilde{I}_j)=0$ if $i\not = j$. If $\tilde{I}_i \not \cong S_i$, then $\Ext^1_A(\tilde{I}_i,\tilde{I}_i)=0.$
	
	\item \label{lem:simproj}
		If $i \not = j$ , then we have $\Ext^1_A(\tilde{P}_i,S_j)=0.$
	
	\item \label{lem:siminj}
		Dually, if $i \not = j$, $\Ext^1_A(S_i,\tilde{I}_j)=0.$
	\end{enumerate}
	\end{lemma}

	\begin{proof} 
	
	We will prove \cref{lem:radproj} and \cref{lem:simproj}, since the other results are dual.
	\vsp
	
		\cref{lem:radproj} Fix a vertex $i$. Let $P_i=Ae_i$ be the indecomposable projective $A$-module at vertex $i$. Notice that  $\tilde{P}_i \cong P_i/J P_i$ as $A$-modules. Let $a_1, \ldots, a_N$ be a complete list of the loops in $Q$ with source $i$, and let $x_1, \ldots, x_M$ be a complete list of arrows that are not loops with source $i$ in $Q$. Note that for $1 \leq m \leq M, 1\leq n \leq N$, $\kk x_m \cong S_{{t(x_m)}}$ and $\kk a_n \cong S_i.$ Then the following is a projective resolution of $\tilde{P}_i$ in $A$.
		\[
		\begin{tikzcd}
			\bigoplus_{n=1}^N\l(\bigoplus_{m=1}^M P_{t(x_m)} \bigoplus P_i^{\oplus N}\r) \arrow[r, "\psi"] \arrow[d] & P_i^{\oplus N} \arrow[r, "\varphi"] \arrow[d] & P_i \arrow[r] & \tilde{P}_i \arrow[r]  & 0 \\
			\bigoplus_{n=1}^N \l(\bigoplus_{m=1}^M \kk x_m \bigoplus_{l=1}^N \kk a_l\r) \arrow[hookrightarrow]{ru} & \bigoplus_{n=1}^N \kk a_n \arrow[hookrightarrow]{ru}  &  & &
		\end{tikzcd}
		\]
	
		We need to take first homology of the following sequence:
		$$0 \rightarrow \Hom_A(P_i, \tilde{P}_j) \xrightarrow{\varphi^*} \bigoplus_{n=1}^N \Hom_A(P_i, \tilde{P}_j) \xrightarrow{\psi^*} \bigoplus_{n=1}^N \Hom_A\l(\bigoplus_{m=1}^M P_{t(x_m)} \bigoplus P_i^{\oplus N}, \tilde{P}_j\r).$$
		
		\vsp
	
		\emph{Case 1: $i \not = j$.}
		
		We will show that $\Hom_A(P_i, \tilde{P_j})=0;$ therefore, $\Ext^1_A(\tilde{P}_i,\tilde{P}_j) =0$.
		
		By \cref{prp:onetooneJ}, $\Hom_A(\tilde{P}_i,\tilde{P}_j) \cong \Hom_B(\tilde{P}_i, \tilde{P}_j).$ Then
		\[
		\dim \Hom_A(P_i, \tilde{P}_j) = \dim e_i \tilde{P}_j = \dim e_i B e_j = \dim \Hom_B(\tilde{P}_i, \tilde{P}_j) = \dim \Hom_A(\tilde{P}_i,\tilde{P}_j) =0,
		\]
		where the last equality holds because $\tilde{P}_i, \tilde{P}_j$ are in the same brick set.
		
		\vsp
	
		\emph{Case 2: $i=j$.}
		
		Notice that $\psi^*=\oplus_{n=1}^N \psi_n^*$, where $\psi_n^*: \Hom_A(P_i, \tilde{P}_j) \rightarrow \Hom_A\l(\bigoplus_{m=1}^M P_{t(x_m)} \bigoplus P_i^{\oplus N}, \tilde{P}_j\r).$ We will show that $\ker \psi_n^*=0$ for each $n$. Therefore, $\Ext^1_A(\tilde{P}_i,\tilde{P}_j)=0$.  
		
		$$\ker \psi_n^* = \l\{ g: P_i \rightarrow \tilde{P}_j : \quad \l[ \bigoplus_{m=1}^M P_{t(x_m)} \bigoplus_{l=1}^N P_i \xrightarrow{\psi_n} P_i \xrightarrow{g} \tilde{P}_j \r]=0 \r\} $$
	
		Since $\im \psi_n = \bigoplus_{m=1}^M \kk x_m \bigoplus_{l=1}^N \kk a_l,$ we have
	
		$$\ker \psi_n^* =  \l\{ g: P_i \rightarrow \tilde{P}_j :  \quad g(x_m)=0=g(a_l), \quad 1 \leq m \leq M,  1 \leq l \leq N \r\} $$
	
		Let $g \in \ker \psi_n^*.$ Since as a vector space, we have 
		$$\tilde{P}_i=\frac{\kk e_i \oplus \kk x_1 \oplus \cdots \oplus \kk x_M \oplus \kk a_1 \oplus \cdots \oplus \kk a_N}{\kk a_1 \oplus \cdots \oplus \kk a_N},$$
		we have that
		$$g(e_i) = [\alpha e_i + \beta_1 x_1 + \cdots + \beta_N x_M]$$
		for some $\alpha, \beta_1, \ldots, \beta_M \in \kk $. 
			
		We can see that $[\beta_1 x_1 + \cdots + \beta_M x_M]=0$ by noting that since none of the $x_k$ have target $i$,
		$$[\alpha e_i]=e_i[\alpha e_i +\beta_1 x_1 + \cdots + \beta_M x_M]=e_ig(e_i)=g(e_i)=[\alpha e_i +\beta_1 x_1 + \cdots + \beta_M x_M].$$
	
		There exists at least one non-loop arrow $x_j$ with source $i$ because $\tilde{P}_i \not \cong S_i$. Hence 
		$$0=g(x_j)=x_jg(e_i)=x_j[\alpha e_i]=[\alpha x_j] \Rightarrow \alpha=0.$$
	
		Then $g(e_i)=0,$ so $g=0$ and hence $\ker \psi_n^*=0$.
		
		\vsp

 		\cref{lem:simproj} We will use the same projective resolution of $\tilde{P}_i$ as the proof of \cref{lem:radproj} and the same notation.
	
		\[
		\begin{tikzcd}
			\bigoplus_{n=1}^N\l(\bigoplus_{m=1}^M P_{t(x_m)} \bigoplus P_i^{\oplus N}\r) \arrow[r, "\psi"] \arrow[d] & \bigoplus_{n=1}^N P_i \arrow[r, "\varphi"] \arrow[d] & P_i \arrow[r] & \tilde{P}_i \arrow[r]  & 0 \\
			\bigoplus_{n=1}^N \l(\bigoplus_{m=1}^M \kk x_m \bigoplus_{l=1}^N \kk a_l\r) \arrow[hookrightarrow]{ru} & \bigoplus_{n=1}^N \kk a_n \arrow[hookrightarrow]{ru}  &  & &
		\end{tikzcd}
		\]
	
		Apply $\Hom_A(-, S_j):$
		$$0 \rightarrow \Hom_A(P_i, S_j) \xrightarrow{\varphi^*} \bigoplus_{n=1}^N \Hom_A(P_i, S_j) \xrightarrow{\psi^*} \bigoplus_{n=1}^N \Hom_A\l(\bigoplus_{m=1}^M P_{t(x_m)} \bigoplus P_i^{\oplus N}, S_j\r).$$
	
		Then $\Ext^1_A(\tilde{P}_i,S_j)=0$ since $\Hom_A(P_i, S_j)=0$ if $i \not = j$ (\cref{fct:collection}). 
		\end{proof}
		
				\begin{lemma}\label{lem:ext1projnohom}
		Assume \cref{hyp:quiver}. Let $\tilde{P_i}$ denote the indecomposable projective $B$-module at vertex $i$. Suppose that $M$ is an $A$-module such that $\Hom_A(P_i,M)=0.$ Then 
		\[
		\Ext^1_A(\tilde{P_i},M)=0.
		\]
		\end{lemma}
		
		\begin{proof}
		Let $a_1, \ldots , a_N$ be a complete list of loops at $i$. First we construct a projective resolution of $\tilde{P_i}$. As a vector space,
		\[ 
		\tilde{P_i}\cong\frac{P_i}{\kk a_1 \oplus \cdots \oplus \kk a_N}.
		\]
		We have a projective resolution:
		\[
		\label{res:A4_12}
		\begin{tikzcd}
		\bigoplus_{i=1}^N  P_i \arrow[r, "\varphi"] \arrow[d] & P_i \arrow[r, "\pi"] & 	\tilde{P_i} \arrow[r] & 0 \\
		\kk a_1 \oplus \cdots \oplus \kk a_N \arrow[hookrightarrow]{ru} & {} & {} & {}
		\end{tikzcd}
		\]

		Now take $\Hom_A\l(-,M\r)$:
		\[
		\begin{tikzcd}
		0 \arrow[r] & \Hom_A\l(P_i,  M\r) \arrow[r, "\varphi^*"] & 			\bigoplus_{l=1}^L \Hom_A\l(P_i,  M\r)
		\end{tikzcd}
		\]

		Since $\Hom_A\l(P_i, M\r)=0,$ 
		\[
		\Ext^1_A \l( \tilde{P_i}, M \r)=0.
		\]
		\end{proof}
		
		\begin{definition}
		A vertex $i$ in a quiver is called a sink if there are no arrows $\alpha$ with $s(\alpha)=i$.
		\end{definition}
		
		\begin{lemma} \label{lem:sinknohom}
		Assume \cref{hyp:quiver}. Let $S_i$ denote the simple $A$-module at vertex $i$. If $i$ is a sink in the quiver $Q'$ describing $B$, and if $M$ is an $A$-module such that $\Hom_A(P_i, M)=0$, then 
		\[
		\Ext^n_A(S_i, M)=0, \: \forall n.
		\]
		\end{lemma}
		

		\begin{proof}
		
		Let $P_i$ denote the indecomposable projective at vertex $i$ in $A$. Let $a_1, \ldots, a_N$ denote the loops with source $i$. If $i$ is a sink in $Q'$, then as a vector space, $\ker ( P_i \rightarrow S_i) = \kk a_1 \oplus \cdots \oplus \kk a_N$. We have the following projective resolution of $S_i$:
		
\[
\begin{tikzcd}
\cdots \arrow[r] & (P_i^{\oplus N})^{\oplus N^2} \arrow[r] \arrow[d] & (P_i^{\oplus N})^{\oplus N} \arrow[r] \arrow[d] & P_i^{\oplus N} \arrow[r] \arrow[d] & P_i \arrow[r, "\pi"] & S_i \arrow[r] & 0 \\
{} & (\kk a_1 \oplus \cdots \oplus \kk a_N)^{\oplus N^2} \arrow[hookrightarrow]{ru} & (\kk a_1 \oplus \cdots \oplus \kk a_N)^{\oplus N} \arrow[hookrightarrow]{ru} & \kk a_1 \oplus \cdots \oplus \kk a_N 	\arrow[hookrightarrow]{ru} & {} & {} & {}
\end{tikzcd}
\]

		We must take $\Hom_A(-,M)$ of this resolution. But since $\Hom_A(P_i,M)=0$, all the terms are $0$. So $\Ext^n_A(S_i, M)=0$ for any $n$.
		
		\end{proof}

		The following lemma comes in handy when calculating $\Ext^1$ groups.
				
		\begin{lemma} \label{lem:kerpsi}
		Assume \cref{hyp:quiver}. Let $M$ be an $A$-module which is annihilated by all of the loops in $A$, and let $P_i$ be the indecomposable projective $A$-module at vertex $i$. Let $\{L_1, \ldots, L_n\}$ be a complete list of loops of $A$ with source $i$. Then for any subset $\{i_1, \ldots, i_k\} \subset \{1, \ldots, n\}$,
		\[
		\l\{ g: P_i \rightarrow M : g(L_{i_1})=\cdots=g(L_{i_k})=0 \r\} = \Hom_A(P_i,M).
		\]
		\end{lemma}
		
		\begin{proof}
		For any $g:P_i \rightarrow M$, $g(L_j)=L_jg(e_i)=0,$	 since $L_j$ annihilates $M$.
		\end{proof}
		
		\vsp
		
		
We use the following proposition to show that certain indecomposable modules cannot be bricks.
		
		\begin{proposition}\label{prp:notindifannih}
		Let $Q$ be a finite quiver with vertices $i,j,k$ with arrows $\alpha: i \rightarrow j, \: \beta: j \rightarrow k$. Suppose that for any arrow $\gamma$, $j \not = s(\gamma)$ except when $\gamma=\beta$, and $j \not=t(\gamma)$ except when $\gamma=\alpha$. If $V=(V_m, f_{\gamma})$ is a representation of $Q$ such that $V_i=\kk^l, V_j=\kk^m, V_k=\kk^n$ for $l,n \geq 1$ and $m \geq 2$, $f_{\alpha} \not=0\not= f_{\beta},$ and $f_{\beta}\circ f_{\alpha}=0$, then $V$ is decomposable.
		\end{proposition}
		
		\begin{proof}
		Notice that removing the vertex $j$ creates two disjoint quivers $Q', Q''$, where $Q'$ contains vertex $i$ and $Q''$ contains vertex $k$. Let $V(Q'), V(Q'')$ denote the set of vertices of $Q',Q''$ respectively.
		
		First, here is a general picture of $Q$:
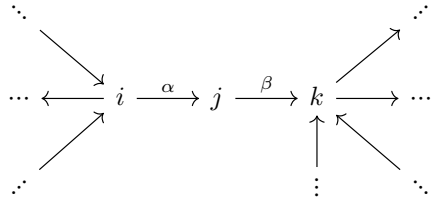
		\[
		\begin{tikzcd}
		\ddots \arrow[rd] & {} & {} & {} & \udots \\
		\cdots & i \arrow[l] \arrow[r, "\alpha"] & j \arrow [r, "\beta"] & k \arrow [ru] \arrow[r] & \cdots \\
		\udots \arrow[ru] & {} & {} & \vdots \arrow[u] & \ddots \arrow [lu]
		\end{tikzcd}
		\]
		Here is a general picture of $V=(V_a, f_{\gamma})$ (for simplicity, denote $f_{\alpha}$ by $f$ and $f_{\beta}$ by $g$):
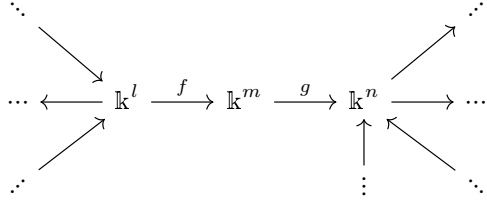
		\[
		\begin{tikzcd}
		\ddots \arrow[rd] & {} & {} & {} & \udots \\
		\cdots & \kk^l \arrow[l] \arrow[r, "f"] & \kk^m \arrow [r, "g"] & \kk^n \arrow [ru] \arrow[r] & \cdots \\
		\udots \arrow[ru] & {} & {} & \vdots \arrow[u] & \ddots \arrow [lu]
		\end{tikzcd}
		\]
		Since $\im f \not =0$, $g\not=0$, and $g \circ f =0,$ we may choose a basis $\{e_1, \ldots, e_p, e_{p+1}, \ldots, e_m\}$ for $\kk^m$ such that $1 \leq p \leq m-1$ and
		\begin{align*}
		g(e_i)&=0 \text{ for }i=1, \ldots,p \\
		g(e_j)&\not=0 \text{ for }j=p+1, \ldots, m.
		\end{align*}
		
		Since $g \circ f=0, \im f \subset \kk e_1 \oplus \cdots \oplus \kk e_p$. Define a map of representations $\varphi: V \rightarrow V$, where for any $a_1, \ldots, a_m \in \kk$,
		\begin{align*}
		\varphi_a &= \Id_{V_a} \text{ if }a\in V(Q') \\
		\varphi_a &= 0 \text{ if } a \in V(Q'') \\
		\varphi_j&:\kk^m \rightarrow \kk^m \\
		&(a_1, \ldots, a_m)  \mapsto (a_1, \ldots, a_p, 0, \ldots, 0).
		\end{align*}
	
		Since $\varphi$ is a non-trivial idempotent in the endomorphism ring of $V$, $V$ is decomposable by \cite[Cor. 1.7.3]{Drozd}.	
		\end{proof}
	
		\vsp
		

	\section{Results for radical square zero algebras}
	
	We will apply \cref{lem:quivpowerful} to calculate the Frobenius-Perron dimension of the category of representations of a radical square zero bound quiver algebra. To prove \cref{lem:quivpowerful}, we need some additional results.
	
	\begin{lemma} \label{lem:powerful}
		Let $A$ be a finite-dimensional algebra over a field $\kk $. Then if $V$ is a brick object in $A\dash\bmod$ and $x \in Z(A)$, there exists $\lambda \in \kk $ such that for all $v\in V$, $x.v=\lambda v$.
	\end{lemma}

	\begin{proof}
		We will prove the contrapositive. Notice that if $x$ is a central element of $A$, then we have an induced endomorphism of any $A$-module $V$ defined by:
		$$\varphi_x: V \rightarrow V, \quad \varphi_x(v)=x.v.$$
		
		Suppose that there exists a central element $x$ of $A$ and a representation $V$ such that $\varphi_x \not=\lambda \Id_V$ for any $\lambda \in \kk $. Then $\varphi_x$ and $\Id_V$ are linearly independent elements of $\Hom_A(V,V),$ so $V$ is not a brick.
	\end{proof}
	
	\begin{proposition} \label{prp:annih}
		If $x$ is a central nilpotent element of $A$ and $V$ is a brick module, then the ideal $(x)$ annihilates $V$.
	\end{proposition}
	
	\begin{proof}
		By \cref{lem:powerful}, for all $v\in V$ we must have $x.v=\lambda v$ for some $\lambda \in \kk $. Suppose $x^n=0$. Then $0=x^n.v=\lambda^n v$ for every $v\in V$. Hence $\lambda =0$ and $x$ annihilates $V$. Since $x$ is central, the two-sided ideal generated by $x$ must also annihilate $V$.	
	\end{proof}
	
	\begin{corollary} \label{cor:annih}
		Let $J$ be an ideal generated by central nilpotent elements of $A$. If $V$ is a brick $A$-module, then $J$ annihilates $V$.
	\end{corollary}
			
	\begin{proposition} \label{prp:onetooneJ}
		Let $J$ be an ideal generated by central nilpotent elements. Then there is a one-to-one correspondence between the isomorphism classes below:
		$$\{ \text{brick } A\dash modules \} \longleftrightarrow \{ \text{brick } A/J \dash modules \}$$
		Furthermore, if $M,N$ are brick $A/J$-modules, then we have a natural isomorphism
		\[
		\Hom_A(M,N) \cong \Hom_{A/J}(M,N).
		\]
		Specifically, if we define $J_{nc}$ to be the ideal generated by all central nilpotent elements, we get the one-to-one correspondence and natural isomorphism above.
	\end{proposition}
	
	\begin{proof}
		Suppose that $M,N$ are brick $A$-modules. Then $J$ annihilates both $M$ and $N$ by \cref{cor:annih}. By \cref{fct:collection} \cref{prp:equiv} we have a natural isomorphism
		\[
		\Hom_A(M,N) \cong \Hom_{A/J}(M,N).
		\]
		Therefore $M,N$ are also brick $A/J$-modules.
		
		Now suppose that $L,P$ are brick $A/J$-modules. Then $L,P$ can be thought of as $A$-modules via \cref{fct:collection} \cref{prp:equiv}, and we have
		\[
		\Hom_{A/J}(L,P) \cong \Hom_A(L,P).
		\]
		Therefore $L,P$ are also brick $A$-modules.
	\end{proof}
	
	\begin{proof}[Proof of \cref{lem:quivpowerful}] 
	 First suppose that $L,P$ are brick $C$-modules annihilated by $(\geq 2)$. Then by \cref{fct:collection} \cref{prp:equiv} we have natural isomorphisms
	 \[
	 \Hom_C(L,P) \cong \Hom_{C/(\geq 2)}(L,P) \cong \Hom_B(L,P) = \Hom_{A/J}(L,P) \cong \Hom_A(L,P).
	 \]
	 Therefore $L,P$ are brick $A$-modules.
	 
	 Now suppose that $M,N$ are brick $A$-modules. Then we have natural isomorphisms
	 \begin{align*}
	 \Hom_A(M,N) & \cong \Hom_{A/J}(M,N) \quad (\text{\cref{prp:onetooneJ}}) \\
	 & =\Hom_B(M,N) \\
	 & \cong \Hom_{C/(\geq 2)}(M,N) \\
	 & \cong \Hom_C(M,N) \quad (\text{\cref{fct:collection} \cref{prp:equiv}}).
	 \end{align*}
	 Therefore, $M,N$ are brick $C$-modules.
	\end{proof}
	
	\vsp


\section{An example of arbitrarily large irrational Frobenius-Perron dimension}
	
	\begin{example} \label{ex:Qnm} Consider the following quiver $Q(n,m)$ with $n$ loops at vertex $1$ and $m$ loops at vertex $2$: 
		
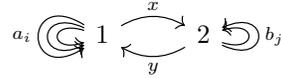
\[
		\begin{tikzcd}
			1  \arrow[out=155, in=205, loop] \arrow[out=162, in=198, loop]  \arrow[out=148, in=212, loop, swap, "a_i"] \arrow[r,bend left,"x"]  & 2 \arrow[l,bend left,"y"] \arrow[out=335, in=25, loop, swap, "b_j"] \arrow[out=342, in=18, loop]						\end{tikzcd} 
		\]
	
	Let $A=\kk Q(n,m)/(\geq 2).$ We want to calculate $\fpd \l(A\dash \bmod\r)$. We can see that $J=(a_1, \ldots, a_n, b_1, \ldots, b_m)$ is an ideal generated by central nilpotents. Let $B:=A/J$. Then $B$ is described by the following quiver with relations $(\geq 2)=0$:
	\[
	\begin{tikzcd}
		1  \arrow[r,bend left,"x"]  & 2 \arrow[l,bend left,"y"]			
	\end{tikzcd}
	\]
	
	By \cref{lem:quivpowerful}, we have a one-to-one correspondence between bricks of $A$ and bricks of $B$, as well as the following natural isomorphism for any $A$-modules $M,N$:
	\[
	\Hom_{A}(M,N) \cong \Hom_{B}(M,N).
	\]
	Hence it suffices to compute the dimensions of $\Hom_{A}$-spaces in $B$.
	
	The representation theory of $B$ is well-known \cite{CGWZZZ}. There are 4 non-isomorphic indecomposable modules. First, we have the two simples $S_1, S_2$, which are described by $S_i=Be_i/(\geq 1)$. Then we have the two indecomposable projectives $M_1,M_2$, where $M_i=Be_i$. $S_1,S_2$ are bricks since they are simples. We have $\Hom_{B}(M_i,M_i) =\Hom_{B}(Be_i, Be_i)\cong e_iBe_i \cong ke_i$, so each $M_i$ is a brick as well. Under the one-to-one correspondence, these map to bricks in $A$ where $S_i=Be_i/(\geq 1)=Ae_i/(\geq 1),$ which are simple in $A$, and $M_i=Be_i=Ae_i/Je_i$. 
	
	We know that $\Hom_A(M_i,M_j) \cong \Hom_{B}(M_i,M_j)=\Hom_{B}(Be_i,Be_j)\cong e_jBe_i,$ which is non-zero for $i\not= j$. So $M_1, M_2$ cannot be in the same brick set. Furthermore, we have the following short exact sequences:
	$$0 \rightarrow S_1 \rightarrow M_2 \rightarrow S_2 \rightarrow 0, \quad \quad
	0 \rightarrow S_2 \rightarrow M_1 \rightarrow S_1 \rightarrow 0.$$
	
	Hence we cannot have a brick set containing both $S_i, M_j$ for any $i,j$.
	
	Therefore there are only 5 possible brick sets:
	$\phi_{1}=\{M_1\},$ $\phi_2=\{M_2\}$, $\phi_3=\{S_1,S_2\},$ $\phi_{4}=\{S_1\},$ and $\phi_5=\{S_2\}.$ Since $\phi_4,\phi_5$ are subsets of $\phi_3$, $\rho(A(\phi_{(4 \text{ or }5)}))\leq \rho(A(\phi_3))$, so we do not need to consider them. We can compute the following.
	
	$\rho(A(\phi_1))=\Ext^1(M_1,M_1)=0$ by \cref{lem:rad}.
	
	$\rho(A(\phi_2))=\Ext^1(M_2,M_2)=0$ by \cref{lem:rad}.
	
	$\rho(A(\phi_3))=\rho 	\begin{pmatrix} n & 1 \\
								1 & m 
						\end{pmatrix} = \frac{1}{2}\l(\sqrt{(m-n)^2+4}+m+n\r)$ by \cref{fct:collection} \cref{itm:dimextsimples}.
	
	Then 
	$$\fpd \l(A\dash \bmod\r) = \sup \{ \rho(A(\phi_i)): 1\leq i \leq 3 \}= \frac{1}{2}\l(\sqrt{(m-n)^2+4}+m+n\r).$$ 
	
	\begin{proof}[Proof of \cref{prp:Qnm}] 
	We can get arbitrarily large irrational numbers using this example. Let $r \in \R^{\geq 0}$. Let $N$ be the smallest integer greater than or equal to $r$. Let $n=N, m=N+1.$ Then we have
	$$\fpd \l(A\dash\bmod\r) = \frac{1}{2}\l(\sqrt{(N+1-N)^2+4}+N+1+N\r)=\frac{1}{2}\l(2N+\sqrt{5}+1\r)= N + \frac{1+\sqrt{5}}{2} > N \geq r.$$
	\end{proof}
	
	\begin{question}
	What is $\fpd D^b(A\dash\bmod)$? Is it still irrational in the cases where $\fpd \l(A\dash\bmod\r)$ is irrational?
	\end{question}
	
	\end{example}

	\vsp
	

\section{Frobenius-Perron dimension of modified ADE bound quiver algebras}

	\subsection{$A(n)$ algebras}
	
	\begin{theorem}\label{thm:A_n1dir}
	
	Let $A(n)$ be the finite dimensional radical square zero algebra described by the following quiver $Q$ with $N_i$ loops (labeled $a_i^l$ for $l=1, \ldots, N_i$) at vertex $i$:
	
	\[
	\begin{tikzcd}
		1  \arrow[out=252, in=-72, loop] \arrow[out=245, in=-65, loop] \arrow[out=238, in=-58, loop, swap, "a_1^l"] \arrow[r, "x_2"] & 2 \arrow[out=245, in=-65, loop, swap, "a_2^l"] \arrow[r, "x_3"] &  \cdots \arrow[r, "x_{n}"] & n \arrow[out=245, in=-65, loop] \arrow[out=238, in=-58, loop, swap, "a_n^l"]
	\end{tikzcd} 
	\]	
	
	Then 
	\[
	\fpd \l( A(n)\dash\bmod \r)= \max\{N_1, \ldots, N_n\}.
	\]
	\end{theorem}
	
	\vsp
	
	\begin{proof}
	
	We have $A(n)=\kk Q/(\geq 2).$ Let $J$ be the ideal generated by the loops. Let $B:=A(n)/J.$ In other words, $B$ is the radical square zero algebra described by the following quiver $Q'$:
	
	\[
	\begin{tikzcd}
		1   \arrow[r, "x_2"] & 2 \arrow[r, "x_3"] &  \cdots \arrow[r, "x_{n}"] & n 
	\end{tikzcd} 
	\]	
	
	We can see that $B=C/(\geq 2),$ where $C:=\kk Q'$.
	
	By \cref{lem:quivpowerful}, we have a one-to-one correspondence between bricks of $A(n)$ and bricks of $C$ annihilated by $(\geq 2)$, as well as the following natural isomorphism for any brick $A(n)$-modules $M,N$:
	\[
	\Hom_{A(n)}(M,N) \cong \Hom_{B}(M,N) \cong \Hom_{C}(M,N).
	\]
	Hence it suffices to compute the dimensions of $\Hom_{A(n)}$-spaces in $B$ or $C$.
	
	The set of brick modules of $C$ is a subset of the set of indecomposable modules of $C,$ which are well known by Gabriel's theorem \cite{Gabriel}. The dimension vectors of these modules correspond precisely to the positive roots for the $A_n$ root system \cite[p. 64]{Humphreys}.
	
	Let $E_i:=(0,\ldots,1, \ldots, 0)$, the $n$-dimensional vector with a $1$ in the $i$th position and $0$ elsewhere. Then the positive roots of $A_n$ are given by the following vectors:
	\[
	\l\{\sum_{l=i}^j E_l : 1 \leq i \leq j \leq n \r\},
	\]
	
	or in different notation,
	
	\begin{align*}
	&(1,0,\ldots, 0) \\
	&(0,1,\ldots,0) \\
	& \vdots \\
	&(0,\ldots,0,1) \\
	&(1,1, 0, \ldots, 0) \\
	&(0,1,1,0, \ldots, 0) \\
	& \vdots \\
	&(0, \ldots, 0,1,1) \\
	&(1,1,1,0,\ldots, 0) \\
	&(0,1,1,1,0,\ldots,0) \\
	& \vdots \\
	&(0, \ldots, 0,1,1,1) \\
	&\vdots \\
	&(1,\ldots,1).
	\end{align*}
	
	However by \cref{fct:collection} \cref{prp:notannihgeqn}, \cref{prp:indreps}, any dimension vector with more than two $1$'s in a row gives a $C$-module that is not annihilated by $(\geq 2).$ Therefore, the only indecomposables that could possibly be bricks of $A(n)$ are the following. Note that below we use Schiffler's notation for representations \cite[Remark 1.1]{Sch}.
	
	\begin{align*}
	(1)=&(1,0,\ldots, 0) \\
	(2)=&(0,1,\ldots,0) \\
		& \vdots \\
	(n)=&(0,\ldots,0,1) \\
	\l(\gfrac{1}{2}\r)=&(1,1, 0, \ldots, 0) \\
	\l(\gfrac{2}{3}\r)=&(0,1,1,0, \ldots, 0) \\
				& \vdots \\
	\l(\gfrac{n-1}{n}\r)=&(0, \ldots, 0,1,1).
	\end{align*}
		
	We will show that all of these are bricks of $A(n)$ and calculate the $\Hom$ matrix.
	
	Because $(i),(j)$ are simple,
	\begin{equation}
	\dim \Hom_{C}((i),(j))=\delta_{ij},
	\end{equation}
	so $(i),(j)$ are bricks.
	
	If we have $\pgfrac{i}{i+1}, \pgfrac{j}{j+1}$ with $i \not=j, i\not=j+1,$ and $i\not=j-1$, then we have no overlap between the non-zero vector spaces of the representations, and $\Hom_{C}\l(\pgfrac{i}{i+1}, \pgfrac{j}{j+1}\r)=0.$ So we only need to check the morphisms in the cases $i=j, i=j+1, i=j-1$. 
	
	Denote the indecomposable projective (respectively injective, respectively simple) of $B$ at vertex $i$ by $\tilde{P}_i$ (respectively $\tilde{I_i}$, respectively $\tilde{S_i}$). Then by \cref{fct:collection} \cref{prp:projinjrepQ},
	\begin{equation}
	\tilde{P_i}=\l( \gfrac{i}{i+1} \r)=\tilde{I}_{i+1}.
	\end{equation}
	
	If $i=j$, 
	\[
	\dim \Hom_{A(n)}\l( \pgfrac{i}{i+1}, \pgfrac{i}{i+1}\r)=\dim \Hom_{B}\l( \tilde{P_i}, \tilde{P_i}\r)=\dim \Hom_{B}\l( Be_i, Be_i\r)=\dim e_iBe_i=1,
	\]
	so $ \pgfrac{i}{i+1}$ is a brick.
	\vsp
	
	If $i=j+1,$ 
	\[
	\dim \Hom_{A(n)}\l( \pgfrac{i}{i+1}, \pgfrac{i-1}{i}\r)=
	\dim \Hom_{B}\l( \tilde{P_i}, \tilde{P}_{i-1}\r)=
	\dim \Hom_{B}(Be_i, Be_{i-1})=
	\dim e_{i}Be_{i-1}=
	1.
	\]
	
	If $i=j-1$,
	\[
	\Hom_{A(n)}\l( \pgfrac{i}{i+1}, \pgfrac{i+1}{i+2}\r)=
	\Hom_{B}\l( \tilde{P_i}, \tilde{P_{i+1}}\r)=
	\Hom_{B}(Be_i, Be_{i+1})=
	e_{i}Be_{i+1}=
	0.
	\]
	
	Therefore,
	\begin{equation}
	\dim \Hom_{A(n)}\l(\pgfrac{i}{i+1}, \pgfrac{j}{j+1} \r) = \begin{cases}
											1 & \text{ if }i=j+1 \\
											1 & \text{ if }i=j \\
											0 & \text{ otherwise}
											\end{cases}.
	\end{equation}
	
	Recall that by \cref{fct:collection} \cref{prp:dimhoms1}, $\dim \Hom_{B}(\tilde{S}_i,\tilde{I}_j)=\delta_{ij}$. Therefore since $\tilde{S}_i=(i), \tilde{I}_{j+1}=\pgfrac{j}{j+1},$

	\begin{equation}
	\dim \Hom_{A(n)}\l((i), \l(\gfrac{j}{j+1}\r)\r)=	\begin{cases}
										1 &\text{ if } i=j+1 \\
										0 &\text{ otherwise}
										\end{cases}.
	\end{equation}
	
	Now let us consider $\l(\gfrac{i}{i+1}\r), (j)$. By \cref{fct:collection} \cref{prp:dimhoms1},
	\begin{equation}
	\dim \Hom_{A(n)}\l(\l(\gfrac{i}{i+1}\r),(j)\r)=\dim \Hom_{B}(\tilde{P_i}, \tilde{S_j})=\delta_{ij}.
	\end{equation}
	
	Compiling all this information, we have the following block matrix $H$. We define $H_{ij}=\dim \Hom_{A(n)}(X_i,X_j)$, where $X_i,X_j \in \l\{ (1), \ldots, (n), \pgfrac{1}{2}, \ldots, \pgfrac{n-1}{n}\r\}$:
	
	\[
	H=
  	\begin{blockarray}{ccccccc|ccccc}
    	{} & (1) & (2) & (3) & \cdots & (n-1) & (n) & \l(\gfrac{1}{2}\r) & \l(\gfrac{2}{3}\r) & \cdots & \l(\gfrac{n-2}{n-1}\r) & \l(\gfrac{n-1}{n}\r) \\
    	\begin{block}{c(cccccc|ccccc@{\hspace*{5pt}})}
    	(1) & 1 & 0 & 0 & 0 & \cdots & 0 				& 0 & 0 &  \cdots & 0 & 0 \\
    	(2) & 0 & 1 & 0 & 0 & \cdots & 0 				& 1 & 0 &  \cdots & 0 & 0 \\
	(3) & 0 & 0 & 1 & 0 & \cdots & 0					& 0 & 1 & 0 & 0 & 0 \\
    	\vdots&\vdots&\vdots&\ddots&\ddots&\ddots&\vdots	& \vdots & \ddots &  \ddots & 0 & 0 \\
	(n-1) & 0 & 0  & 0 & 0 & 1 & 0 					& 0 & 0 &  0 & 1 & 0 \\
    	(n) &  0 & 0 & 0 & 0 & 0 & 1 					& 0 & 0  & 0 & 0 & 1 \\
    	\cline{1-12}
    	\l(\gfrac{1}{2}\r) & 1 & 0 & 0 & 0 & 0 & 0			& 1 & 0 & \cdots & 0 & 0 \\
    	\l(\gfrac{2}{3}\r) & 0 & 1 & 0 & 0 & 0 & 0			& 1 & 1 & \cdots & 0 & 0 \\
    	\vdots & \vdots & \ddots & \ddots & \ddots & 0		& 0 & 0 & 1 & \ddots & 0 & 0 \\
	\l(\gfrac{n-2}{n-1}\r) & 0 & 0 & 0 & 1 & 0 & 0 		& 0 & 0 & \ddots & \ddots & 0 \\
    	\l(\gfrac{n-1}{n}\r) & 0 & 0 & 0 & 0 & 1 & 0 	 		& 0 & 0 & 0 & 1 & 1 \\
    	\end{block}
  	\end{blockarray}
	\]
	
	
	Now we can calculate the $\Ext^1_{A(n)}$ matrix. We already know what the $\Ext^1$ groups for the simple objects should be. Recall from \cref{lem:rad} that
	\begin{align*}
	\Ext^1_{A(n)}(S_i, \tilde{I_j})&=0, \: \forall i \not=j \\
	\Ext^1_{A(n)}(\tilde{I_i}, \tilde{I_j})&=0, \text{ if }i\not=j \text{ or }\tilde{I_i}\text{ is not simple} \\
	\Ext^1_{A(n)}(\tilde{P_i}, S_j)&=0, \: \forall i \not=j \\
	\Ext^1_{A(n)}(\tilde{P_i}, \tilde{P_j})&=0, \text{ if }i\not=j \text{ or }\tilde{P_i}\text{ is not simple}. \\
	\end{align*}
	
	Let $*$ denote a possible non-zero matrix entry. We have $(i)=S_i$ and $\l(\gfrac{j}{j+1}\r)=\tilde{P_j}=\tilde{I}_{j+1}$. Using the same notation as before, if we define $E_{ij}=\Ext^1_{A(n)}(X_i,X_j)$, we have
	\[
	E=
  	\begin{blockarray}{ccccccc|ccccc}
    	{} & (1) & (2) & (3) & \cdots & (n-1) & (n) & \l(\gfrac{1}{2}\r) & \l(\gfrac{2}{3}\r) & \cdots & \l(\gfrac{n-2}{n-1}\r) & \l(\gfrac{n-1}{n}\r) \\
    	\begin{block}{c(cccccc|ccccc@{\hspace*{5pt}})}
    	(1) & N_1 & 1 & 0 & 0 & \cdots & 0 				& 0 & 0 &  \cdots & 0 & 0 \\
    	(2) & 0 & N_2 & 1 & 0 & \cdots & 0 				& * & 0 &  \cdots & 0 & 0 \\
	(3) & 0 & 0 & N_3 & 1 & \cdots & 0				& 0 & * & \ddots & 0 & 0 \\
    	\vdots&\vdots&\vdots&\ddots&\ddots&\ddots&\vdots	& \vdots & \ddots &  \ddots & 0 & 0 \\
	(n-1) & 0 & 0  & 0 & 0 & N_{n-1} & 1 				& 0 & 0 & 0 & * & 0 \\
    	(n) &  0 & 0 & 0 & 0 & 0 & N_n					& 0 & 0  & 0 & 0 & * \\
    	\cline{1-12}
    	\l(\gfrac{1}{2}\r) & * & 0 & 0 & 0 & 0 & 0			& 0 & 0 & 0 & 0 & 0 \\
    	\l(\gfrac{2}{3}\r) & 0 & * & 0 & 0 & 0 & 0			& 0 & 0 & 0 & 0 & 0 \\
    	\vdots & \vdots & \ddots & \ddots & \ddots & 0& 0 	& 0 & 0 & 0 & 0 & 0 \\		\l(\gfrac{n-2}{n-1}\r) & 0 & 0 & 0 & * & 0 & 0 		& 0 & 0 & 0 & 0 & 0 \\
    	\l(\gfrac{n-1}{n}\r) & 0 & 0 & 0 & 0 & * & 0 	 		& 0 & 0 & 0 & 0 & 0 \\
    	\end{block}
  	\end{blockarray}
	\]
	
	Now suppose that we have a brick set 
	\[
	\phi = \l\{ (i_1), \ldots, (i_M), \l(\gfrac{j_1}{j_1+1} \r),\ldots, \l(\gfrac{j_L}{ j_L+1} \r): i_1 < i_2 < \cdots <i_M, \quad j_1 < j_2 < \cdots < j_L\r\}.
	\]
	
	We will show that for $* \in \{0,1\}$,
	\[
	A(\phi)=
	\begin{blockarray}{ccccccc|ccccc}
    	{} & (i_1) & (i_2) & (i_3) & \cdots & (i_{M-1}) & (i_M) & \l(\gfrac{j_1}{j_1+1}\r) & \l(\gfrac{j_2}{j_2+2}\r) & \cdots & \l(\gfrac{j_{L-1}}{j_{L-1}+1}\r) & \l(\gfrac{j_L}{j_L+1}\r) \\
    	\begin{block}{c(cccccc|ccccc@{\hspace*{5pt}})}
    	(i_1) & N_{i_1} & * & 0 & 0 & \cdots & 0 				& 0 & 0 & 0 & 0 & 0 \\
    	(i_2) & 0 & N_{i_2} & * & 0 & \cdots & 0 				& 0 & 0 & 0 & 0 & 0 \\
	(i_3) & 0 & 0 & N_{i_3} & * & \cdots & 0				& 0 & 0 & 0 & 0 & 0 \\
    	\vdots&\vdots&\vdots&\ddots&\ddots&\ddots&\vdots		& 0 & 0 & 0 & 0 & 0 \\
	(i_{M-1}) & 0 & 0  & 0 & 0 & N_{i_{M-1}} & * 			& 0 & 0 & 0 & 0 & 0 \\
    	(i_M) &  0 & 0 & 0 & 0 & 0 & N_{i_M}					& 0 & 0 & 0 & 0 & 0 \\
    	\cline{1-12}
    	\l(\gfrac{j_1}{j_1+1}\r) & 0 & 0 & 0 & 0 & 0 & 0			& 0 & 0 & 0 & 0 & 0 \\
    	\l(\gfrac{j_2}{j_2+2}\r) & 0 & 0 & 0 & 0 & 0 & 0			& 0 & 0 & 0 & 0 & 0 \\
    	\vdots & 0 & 0 & 0 & 0 & 0 & 0 						& 0 & 0 & 0 & 0 & 0 \\
	\l(\gfrac{j_{L-1}}{j_{L-1}+1}\r) & 0 & 0 & 0 & 0 & 0 & 0 	& 0 & 0 & 0 & 0 & 0 \\
    	\l(\gfrac{j_L}{j_L+1}\r) & 0 & 0 & 0 & 0 & 0 & 0 	 		& 0 & 0 & 0 & 0 & 0 \\
    	\end{block}
  	\end{blockarray}
	\]
	
	Assuming the claim about $A(\phi)$, if there are no simples in $\phi$, $\rho(A(\phi))=0 \leq \max\l\{ N_1, \ldots, N_n \r\}.$ Else,
	\begin{equation} \label{eq:Anproofrho}
	\rho(A(\phi)) = \max\l\{N_{i_1}, \ldots, N_{i_M}\r\} \leq \max\l\{ N_1, \ldots, N_n \r\}.
	\end{equation}
	
	To compute $A(\phi)$, notice that the upper leftmost and bottom rightmost blocks are given by examination of the $\Ext^1$ matrix. To get the upper rightmost and bottom leftmost blocks, it suffices to show that for all $m,l$,
	\[
	\Ext^1_{A(n)}\l((i_m),\l(\gfrac{j_l}{j_l+1}\r) \r)=0=\Ext^1_{A(n)}\l(\l(\gfrac{j_l}{j_l+1}\r),(i_m) \r).
	\]
	
	By examination of the $\Hom$ and $\Ext^1$ matrices, if $(i) \in \l\{(1), \ldots, (n)\r\}$ and $\pgfrac{j}{k} \in \l\{ \pgfrac{1}{2}, \ldots, \pgfrac{n-1}{n}\r\}$ are in the same brick set, then 
	$$\Ext^1_{A(n)}\l((i),\pgfrac{j}{k}\r)=0=\Ext^1_{A(n)}\l(\pgfrac{j}{k},(i)\r).$$

	Therefore,
	\[
	\Ext^1_{A(n)}\l((i_m),\l(\gfrac{j_l}{j_l+1}\r) \r)=0=\Ext^1_{A(n)}\l(\l(\gfrac{j_l}{j_l+1}\r),(i_m) \r).
	\]
	
	To complete the proof, notice that 
	\[
	\tilde{\phi}:=\{(1), (2), \ldots, (n)\}
	\]
	is a brick set. By \cref{eq:Anproofrho} we have 
	\[
	\rho(A(\tilde{\phi}))=\max \l\{N_1, N_2, \ldots, N_n \r\}.
	\]
	
	Let $\Phi_b$ the set of brick sets of $A(n)$-modules. Again by \cref{eq:Anproofrho},
	\[
	\max \l\{N_1, N_2, \ldots, N_n \r\} = \rho(A(\tilde{\phi}))\leq \sup_{\phi \in \Phi_b} \rho(A(\phi)) \leq \max \l\{N_1, N_2, \ldots, N_n \r\}.
	\]
	
	Therefore,
	\[
	\fpd \l(A(n)\dash\bmod\r) = \sup_{\phi \in \Phi_b} \rho(A(\phi))=\max \l\{N_1, N_2, \ldots, N_n \r\}.
	\]	
	\end{proof}

	\vsp

	
	\subsection{$D(n)$ algebras}
	
	\begin{theorem} \label{thm:Dn1dir}
	Let $D(n)$ be the finite dimensional radical square zero algebra described by the following quiver $Q$ with $N_i$ loops (labeled $a_i^l$ for $l=1, \ldots, N_i$) at vertex i:
	
	\[
	\begin{tikzcd}
	{} & {} & {} & n-1 \arrow[out=65, in=115, loop, swap, "a_{n-1}^l"] & {}  \\
	1   \arrow[r, "x_2"]  \arrow[out=252, in=-72, loop] \arrow[out=245, in=-65, loop] \arrow[out=238, in=-58, loop, swap, "a_1^l"] & 2 \arrow[out=245, in=-65, loop, swap, "a_2^l"] \arrow[r, "x_3"] &  \cdots \arrow[r, "x_{n-2}"] & n-2 \arrow[out=252, in=-72, loop] \arrow[out=245, in=-65, loop] \arrow[out=238, in=-58, loop, swap, "a_{n-2}^l"] \arrow[r, "x_n"] \arrow[u, "x_{n-1}"] & n \arrow[out=245, in=-65, loop] \arrow[out=238, in=-58, loop, swap, "a_n^l"]
	\end{tikzcd} 
	\]
	
	Then
	\[
	\fpd \l(D(n)\dash\bmod\r) = \max\{N_1, \ldots, N_n\}.
	\]
	\end{theorem}
		
	\vsp
	
	\begin{proof}
	
	We will reuse some notation from the previous theorem. Let $D(n)=\kk Q/(\geq 2).$ Let $J$ be the ideal generated by the loops, and let $B:=D(n)/J.$ In other words, $B$ is the radical square zero algebra described by the following quiver $Q'$:
	
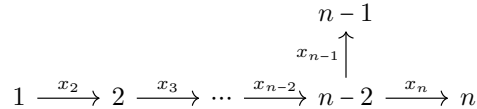
	\[
	\begin{tikzcd}
	{} & {} & {} & n-1 & {}  \\
	1   \arrow[r, "x_2"] & 2 \arrow[r, "x_3"] &  \cdots \arrow[r, "x_{n-2}"] & n-2 \arrow[r, "x_n"] \arrow[u, "x_{n-1}"] & n 
	\end{tikzcd} 
	\]	
	
	We can see that $B=C/(\geq 2),$ where $C:=\kk Q'$.
		
	By \cref{lem:quivpowerful}, we have a one-to-one correspondence between bricks of $D(n)$ and bricks of $C$ annihilated by $(\geq 2)$, as well as the following natural isomorphism:
	\[
	\Hom_{D(n)}(M,N) \cong \Hom_{B}(M,N) \cong \Hom_{C}(M,N).
	\]
	Hence it suffices to compute the dimensions of $\Hom_{D(n)}$-spaces in $B$ or $C$. 
	
	The indecomposable modules of $C$ are well known by Gabriel's theorem \cite{Gabriel}. The dimension vectors of these modules correspond precisely to the positive roots for the $D_n$ root system.
	
	Let $\varepsilon_i$ be the vector in $\R^n$ with a $1$ in the $i$th component and $0$ everywhere else. Recall from \cite[p. 64]{Humphreys} that the root system for $D_n$ consists of the vectors 
	$$\{ \pm (\varepsilon_i \pm \varepsilon_j): \: i \not = j\}.$$ 
	Let 
	$$E_1:= \varepsilon_1-\varepsilon_2, \ldots, E_{n-1}:= \varepsilon_{n-1}-\varepsilon_{n}, E_n:=\varepsilon_{n-1}+\varepsilon_{n}$$ 
	be the standard base for $D_n$. Here we write out the other positive roots in terms of the standard base.
	
	For $1 \leq i < j \leq n,$
\[
\varepsilon_i - \varepsilon_j=E_i + E_{i+1} + \cdots + E_{j-1} =\sum_{l=i}^{j-1}E_l.
\]
	
For $1 \leq i \leq n-2$,
\[	\varepsilon_i + \varepsilon_n = E_i + E_{i+1} + \cdots + E_{n-2} + E_n=\l(\sum_{l=i}^{n-2} E_l \r)+ E_n.
\]

For $1 \leq i \leq n-2$,
\[
\varepsilon_i + \varepsilon_{n-1} =E_i + E_{i+1} + \cdots + E_{n-2} + E_{n-1} + E_n = \sum_{l=i}^n E_l. \]
	

For $1 \leq i < j \leq n-2,$	
\[
\varepsilon_i + \varepsilon_j = (E_i + E_{i+1} + \cdots + E_{n-2})+(E_j+E_{j+1} + \cdots + E_{n-2} + E_{n-1} +E_n) = \sum_{l=i}^{n-2}E_l + \sum_{k=j}^n E_k.
\]
	

		Let $\tilde{P_i}$ (respectively $\tilde{I_i}$) denote the indecomposable projective (respectively injective) of $B$ at vertex $i$. By \cref{fct:collection} \cref{prp:projinjrepQ}, we have
		\begin{align} \label{eq:Dnprojinj1}
		\tilde{P_1}=\tilde{I_2}= \pgfrac{1}{2} 	&=(1,1, 0, \ldots, 0) \\
		\tilde{P_2}= \tilde{I_3} =\pgfrac{2}{3}	&=(0,1,1, 0, \ldots, 0) \\
		&\vdots \\
		\tilde{P}_{n-3}=\tilde{I}_{n-2}=\pgfrac{n-3}{n-2} &=(0, \ldots, 0, 1,1,0,0) \\
		\tilde{I}_{n-1}=\pgfrac{n-2}{n-1} 	&=(0, \ldots, 0,1,1,0) \\
		\tilde{I_n}=\pgfrac{n-2}{n}	  	&=(0, \ldots, 0,1,0,1) \\
		\tilde{P}_{n-2}=\pgfrac{n-2}{n-1 \:\: n} &=(0, \ldots, 0, 1,1,1). \label{eq:Dnprojinj2}
		\end{align}
	
		\begin{proposition} \label{prp:bricksDn}
		The only indecomposable representations of $C$ which are annihilated by $(\geq 2)$ are the following:
		\begin{align*}
		(1) 	&= E_1 =(1, 0, \ldots, 0) \\
		(2) 	&= E_2 = (0,1, 0, \ldots, 0) \\
			&\vdots \\
		(n) 	&=E_n= (0, \ldots, 0, 1) \\
		\pgfrac{1}{2} 	&= E_1+E_2=(1,1, 0, \ldots, 0) \\
		\pgfrac{2}{3}	&=E_2 + E_3= (0,1,1, 0, \ldots, 0) \\
		&\vdots \\
		\pgfrac{n-3}{n-2} &=E_{n-3}+E_{n-2}=(0, \ldots, 0, 1,1,0,0) \\
		\pgfrac{n-2}{n-1} &=E_{n-2}+E_{n-1}=(0, \ldots, 0,1,1,0) \\
		\pgfrac{n-2}{n}	  &=E_{n-2}+E_{n}=(0, \ldots, 0,1,0,1) \\
		\pgfrac{n-2}{n-1 \:\: n} &=E_{n-2}+E_{n-1}+E_{n}=(0, \ldots, 0, 1,1,1).
		\end{align*}
		\end{proposition}

		\begin{proof}

		\emph{Case 0: $E_n=\varepsilon_{n-1}+\varepsilon_n$.}
		
		This representation is annihilated by $(\geq 2)$ because $E_n$ is the simple module at vertex $n$.
		
		\vsp
		
		\emph{Case 1: Roots of the form $\varepsilon_i-\varepsilon_j$ for $1 \leq i < j \leq n$.}
		
		We have
		\[
		\varepsilon_i-\varepsilon_j= E_i + E_{i+1} + \cdots + E_{j-1} = (0,\ldots, 0, 1, \ldots, 1, 0, \ldots, 0),
		\]
		where the first $1$ is in the $i$th position and the last $1$ is in the $(j-1)$th position. Notice that there must be a $0$ in the $n$th position since $j \leq n$. By examination of the graph for $C$, we can see that if $j>i+2$ we have three $1$'s in a row, which corresponds to two non-zero paths of length 1 in a row by \cref{fct:collection} \cref{prp:indreps}. Since the vector spaces involved are 1-dimensional, the composition of these non-zero paths must also be non-zero. By \cref{fct:collection} \cref{prp:notannihgeqn}, if $j>i+2$ the representation is not annihilated by $(\geq 2)$.

		Therefore only dimension vectors of the following form are annihilated by $(\geq 2)$:
		\begin{align*}
		\varepsilon_i-\varepsilon_{i+2}&=E_i + E_{i+1} \text{ for } 1 \leq i \leq n-2 \\
		\varepsilon_i-\varepsilon_{i+1}&=E_i \text{ for } 1 \leq i \leq n-1
		\end{align*}
		
		\vsp
		
		\emph{Case 2: Roots of the form $\varepsilon_i + \varepsilon_n$ for $1 \leq i \leq n-2$.}
		
		We have
		\[
		\varepsilon_i + \varepsilon_n= E_i + E_{i+1} + \cdots + E_{n-2} + E_n=(0, \ldots, 0, 1, \ldots, 1, 0,1),
		\]
		where the first $1$ is in the $i$th position. If $i<n-2$, then we have three one-dimensional vector spaces which sit adjacent to each other in the graph of $C$, which corresponds to a non-zero path of length 2 by \cref{fct:collection} \cref{prp:indreps}. By \cref{fct:collection} \cref{prp:notannihgeqn}, such a representation is not annihilated by $(\geq 2)$. If $i=n-2$, there is no such path.
		
		Then the only representation annihilated by $(\geq 2)$ is the one corresponding to $i=n-2$:
		\[
		E_{n-2}+E_n=(0, \ldots, 0,1,0,1).
		\]
		
		\vsp
		
		\emph{Case 3: Roots of the form $\varepsilon_i + \varepsilon_{n-1}$ for $1\leq i \leq n-2$.}
		
		We have
		\[
		\varepsilon_i + \varepsilon_{n-1}= E_i + E_{i+1} + \cdots + E_{n-2} + E_{n-1}+ E_n=(0, \ldots, 0, 1, \ldots, 1,1),
		\]
		where the first $1$ is in the $i$th position. If $i<n-2$, then we have three one-dimensional vector spaces which sit adjacent to each other in the graph of $C$, which corresponds to a non-zero path of length 2 by \cref{fct:collection} \cref{prp:indreps}. By \cref{fct:collection} \cref{prp:notannihgeqn}, such a representation is not annihilated by $(\geq 2)$. If $i=n-2$, there is no such path.
		
		By \cref{fct:collection} \cref{prp:notannihgeqn}, the only representation annihilated by $(\geq 2)$ is the one corresponding to $i=n-2$:
		\[
		E_{n-2}+E_{n-1}+E_n=(0, \ldots, 0,1,1,1).
		\]
		
		\vsp
		
		\emph{Case 4: Roots of the form $\varepsilon_i+\varepsilon_j$ for $1 \leq i < j \leq n-2$.}
		
		We have
		\begin{align*}
		\varepsilon_i + \varepsilon_j&=(E_i + E_{i+1} + \cdots E_{n-2}) + (E_j + E_{j+1} + \cdots + E_{n-2} + E_{n-1} + E_n) \\
		&= E_i + E_{i+1} + \cdots + E_{j-1} + 2(E_{j}+\cdots + E_{n-2}) + E_{n-1} + E_n.
		\end{align*}
		
		If $j=n-2,$ then
		\begin{align*}
		\varepsilon_i + \varepsilon_j &= E_i + E_{i+1} + \cdots + E_{j-1} + 2(E_{j}+\cdots + E_{n-2}) + E_{n-1} + E_n \\
		&=E_i + \cdots + E_{n-3} + 2E_{n-2}+E_{n-1}+E_n \\
		&=(0, \ldots, 0, 1, \ldots, 1, 2,1,1).
		\end{align*}
		
		If $j\leq n-3$, then
		\begin{align*}
		\varepsilon_i + \varepsilon_j &= E_i + E_{i+1} + \cdots + E_{j-1} + 2(E_{j}+\cdots + E_{n-2}) + E_{n-1} + E_n \\
		&=(0, \ldots, 0, 1, \ldots, 1, 2, \ldots, 2,1,1).
		\end{align*}
		
		By \cref{lem:notDnindreps} below, these representations are not annihilated by $(\geq 2)$. Therefore, no representations of this type are allowed.
		\end{proof}

		\begin{lemma}\label{lem:notDnindreps} 
		Suppose that we have a representation $V$ of $C$ of the form
		\begin{equation} \label{eq:notDnindreps1}
		\begin{tikzcd}
		{} & {} & {} & {} & \kk & {}  \\
		V_1 \arrow[r] & \cdots \arrow[r] & V_{n-4} \arrow[r] & \kk^2 \arrow[r, "f"] & \kk^2 \arrow[r, "h"] \arrow[u, "g"] & \kk
		\end{tikzcd}
		\end{equation}
		or
		\begin{equation}\label{eq:notDnindreps2}
		\begin{tikzcd}
		{} & {} & {} & {} & \kk & {}  \\
		V_1 \arrow[r] & \cdots \arrow[r] & V_{n-4} \arrow[r] & \kk \arrow[r, "f"] & \kk^2 \arrow[r, "h"] \arrow[u, "g"] & \kk
		\end{tikzcd}
		\end{equation}
		for some $\kk$-vector spaces $V_i$ such that $f,g,h$ are non-zero and $g\circ f=0=h\circ f$. Then $V$ is decomposable.
		\end{lemma}
		
		\begin{proof}
		
		Define $Q'$ to be the full subquiver containing the vertices $\l\{1, \ldots, n-3\r\}$, and let $Q''$ be the full subquiver containing the vertices $\l\{ n-1, n\r\}$. Let $V(Q'), V(Q'')$ denote the set of vertices of $Q', Q''$ respectively.
		
		\vsp
		
		We will start with \cref{eq:notDnindreps1}. 
		
		Since $g \not=0$, we can choose a basis $e_1, e_2$ for $V_{n-2}=\kk^2$ such that $g= \pi_1$. Then 
		\begin{align*}
		g \circ f =0 &\Rightarrow \im f \subset \kk e_2 \\
		h \circ f=0, h\not=0 &\Rightarrow h=h_1 \pi_1 \text{ for some non-zero }h_1\in\kk.
		\end{align*}
		
		Then we have a map of representations $\varphi: V \rightarrow V$ such that for any $a_1, a_2 \in \kk$,
		\begin{align*}
		\varphi_a &=0 \text{ if }a\in V(Q') \\
		\varphi_a &= \Id_{V_a} \text{ if }a\in V(Q'') \\
		\varphi_{n-2}&:\kk^2 \rightarrow \kk^2 \\
		&(a_1, a_2) \mapsto (a_1, 0).
		\end{align*}
		
		Since $\varphi$ is a non-trivial idempotent in the endomorphism ring of $V$, $V$ is decomposable by \cite[Cor. 1.7.3]{Drozd}.
		
		\vsp
		
		Now let us prove the assertion for \cref{eq:notDnindreps2}. By choice of basis for $V_{n-3}$ and the fact that $f\not=0$, we can assume that $f=\iota_2$. Since $g\circ f=0=h\circ f$ and $g,h$ are non-zero, we must have $g=g_1 \pi_1, h=h_1 \pi_1$ for some non-zero $g_1,h_1 \in \kk$.
		
		Then we have a map of representations $\varphi: V \rightarrow V$ such that for any $a_1, a_2 \in \kk$,
		\begin{align*}
		\varphi_a &=0 \text{ if }a\in V(Q') \\
		\varphi_a &= \Id_{V_a} \text{ if }a\in V(Q'') \\
		\varphi_{n-2}&:\kk^2 \rightarrow \kk^2 \\
		&(a_1, a_2) \mapsto (a_1, 0).
		\end{align*}
		
		Since $\varphi$ is a non-trivial idempotent in the endomorphism ring of $V$, $V$ is decomposable by \cite[Cor. 1.7.3]{Drozd}.
		\end{proof}
		
		\vsp
		
	The list of indecomposables given in \cref{prp:bricksDn} are in fact bricks of $D(n)$, as can be seen by examination of the following $\Hom$ matrix, where $H_{ij}=\dim \Hom_{D(n)}(X_i,X_j)$ for brick objects $X_i,X_j$ of $D(n)\dash\bmod$:
	
	\[
  	\begin{blockarray}{cccccccc|cccc|ccc}
    	{} & (1) & (2) & \cdots & (n-3) & (n-2) & (n-1) & (n) & \l(\gfrac{1}{2}\r) & \l(\gfrac{2}{3}\r) & \cdots & \l(\gfrac{n-3}{n-2}\r) & \l(\gfrac{n-2}{n-1}\r) & \pgfrac{n-2}{n} & \pgfrac{n-2}{n-1 \: \: n} \\
    	\begin{block}{c(ccccccc|cccc|ccc@{\hspace*{5pt}})}
    	(1) & 1 & 0 & 0 & 0 & 0 & 0 & 0 & 0 & 0 & \cdots & 0 & 0 & 0 & 0  \\
    	(2) & 0 & 1 & 0 & 0 & 0 & 0 & 0 & 1 & 0 & \cdots & 0 & 0 & 0 & 0 \\
	\vdots & 0 & \ddots & \ddots & 0 & \vdots & \vdots & \vdots & 0 & \ddots & 0 & 0 & \vdots & \vdots & \vdots  \\
	(n-3) & 0 & 0 & 0 & 1 & 0 & 0 & 0 & 0 & 0 & \ddots & 0 & 0 & 0 & 0 \\
    	(n-2) & 0 & 0 & 0 & 0 & 1 & 0 & 0 & 0 & 0 & 0 & 1 & 0 & 0 & 0 \\
	(n-1) & 0 & 0  & 0 & 0 & 0 & 1 & 0 & 0 &  0 & 0 & 0 & 1 & 0 & 1 \\
    	(n) &  0 & 0 & 0 & 0 & 0 & 0 & 1 & 0  & 0 & 0 & 0 & 0 & 1 & 1 \\
    	\cline{1-15}
    	\l(\gfrac{1}{2}\r) & 1 & 0 & 0 & 0 & 0 & 0 & 0 & 1 & \cdots & 0 & 0 & 0 & 0 & 0 \\
    	\l(\gfrac{2}{3}\r) & 0 & 1 & 0 & 0 & 0 & 0 & 0 & 1 & 1 & 0 & 0 & 0 & 0 & 0 \\
    	\vdots & \vdots & \ddots & \ddots & \ddots & \vdots & \vdots & \vdots & 0 & \ddots & \ddots & 0 & \vdots & \vdots & \vdots \\
	\l(\gfrac{n-3}{n-2}\r) & 0 & 0 & 0 & 1 & 0 & 0 & 0 & 0 & 0 & 1 & 1 & 0 & 0 & 0 \\
	\cline{1-15}
    	\l(\gfrac{n-2}{n-1}\r) & 0 & 0 & 0 & 0 & 1 & 0 & 0 & 0 & 0 & 0 & 1 & 1 & 0 & 0 \\
	\l(\gfrac{n-2}{n}\r) & 0 & 0 & 0 & 0 & 1 & 0 	 & 0 & 0 & 0 & 0 & 1 & 0 & 1 & 0 \\
	\l(\gfrac{n-2}{n-1 \: \: n}\r) & 0 & 0 & 0 & 0 & 1 & 0 & 0 & 0 & 0 & 0 & 1 & 1 & 1 & 1 \\
    	\end{block}
  	\end{blockarray}
	\]
	
	To calculate the entries of the $\Hom$ matrix, recall from \cref{fct:collection} \cref{prp:dimhoms1}, \cref{prp:dimhoms2} that
	\begin{align*}
	\dim \Hom_{C}(\tilde{P_i},\tilde{S_j})&= \delta_{ij}, \\
	\dim \Hom_{C}(\tilde{S_i},\tilde{I_j})&=\delta_{ij}, \\
	\dim \Hom_{C}(\tilde{I_j},\tilde{I_i})&= \text{ the number of paths from }i\rightarrow j \text { in } C= \dim \Hom_{C}(\tilde{P_j},\tilde{P_i}).
	\end{align*}
	
	Using \cref{eq:Dnprojinj1} through \cref{eq:Dnprojinj2} in conjunction with these facts, we get most of the $\Hom$ matrix. We can calculate the remaining entries:
	\begin{align*}
	&\dim \Hom_{C}\l(\l(i\r),\pgfrac{n-2}{n-1\:\:n} \r) = 	\begin{cases}
											0 & \quad i \in \{1, \ldots, n-2\} \\
											1 & \quad i \in \{n-1, n\}
											\end{cases} \\
	&\dim\Hom_{C}\l(\pgfrac{n-2}{n-1},\l(i\r) \r) 	= 	\begin{cases}
											1 & \quad i =n-2 \\
											0 & \quad \text{otherwise }
											\end{cases} \\
	&\dim\Hom_{C}\l(\pgfrac{n-2}{n},\l(i\r) \r) = 		\begin{cases}
											1 & \quad i =n-2 \\
											0 & \quad \text{otherwise }
											\end{cases} \\
	&\dim\Hom_{C}\l(\pgfrac{n-2}{n-1},\pgfrac{n-2}{n-1\:\:n} \r)=0 \\
	&\dim\Hom_{C}\l(\pgfrac{n-2}{n},\pgfrac{n-2}{n-1\:\:n} \r)=0 \\
	&\dim\Hom_{C}\l(\pgfrac{n-2}{n-1\:\:n}, \pgfrac{n-2}{n-1} \r)=1 \\
	&\dim\Hom_{C}\l(\pgfrac{n-2}{n-1\:\:n}, \pgfrac{n-2}{n} \r)=1.
	\end{align*}

	We have the following $\Ext^1$ matrix using the same notation as before, where $E_{ij}=\dim \Ext^1_{D(n)}(X_i,X_j)$, and $*$ denotes a possible non-zero entry:
	
	\[
  	\begin{blockarray}{cccccccc|cccc|ccc}
    	{} & (1) & (2) & \cdots & (n-3) & (n-2) & (n-1) & (n) & \l(\gfrac{1}{2}\r) & \l(\gfrac{2}{3}\r) & \cdots & \l(\gfrac{n-3}{n-2}\r) & \l(\gfrac{n-2}{n-1}\r) & \pgfrac{n-2}{n} & \pgfrac{n-2}{n-1 \: \: n} \\
    	\begin{block}{c(ccccccc|cccc|ccc@{\hspace*{5pt}})}
    	(1) & N_1 & 1 & 0 & 0 & 0 & 0 & 0 & 0 & 0 & \cdots & 0 & 0 & 0 & 0  \\
    	(2) & 0 & N_2 & \ddots & 0 & 0 & 0 & 0 & * & 0 & \cdots & 0 & 0 & 0 & 0 \\
	\vdots & \vdots & \ddots & \ddots & \ddots & \vdots & \vdots & \vdots & 0 & \ddots & 0 & 0 & \vdots & \vdots & \vdots  \\
	(n-3) & 0 & 0 & 0 & N_{n-3} & 1 & 0 & 0 & 0 & 0 & \ddots & 0 & 0 & 0 & * \\
    	(n-2) & 0 & 0 & 0 & 0 & N_{n-2} & 1 & 1 & 0 & 0 & 0 & * & 0 & 0 & * \\
	(n-1) & 0 & 0  & 0 & 0 & 0 & N_{n-1} & 0 & 0 &  0 & 0 & 0 & * & 0 & * \\
    	(n) &  0 & 0 & 0 & 0 & 0 & 0 & N_n & 0  & 0 & 0 & 0 & 0 & * & * \\
    	\cline{1-15}
    	\l(\gfrac{1}{2}\r) & * & 0 & 0 & 0 & 0 & 0 & 0 & 0 & 0 & \cdots & 0 & 0 & 0 & 0 \\
    	\l(\gfrac{2}{3}\r) & 0 & * & 0 & 0 & 0 & 0 & 0 & 0 & 0 & 0 & 0 & 0 & 0 & 0 \\
    	\vdots & \vdots & \ddots & \ddots & \ddots & \vdots & \vdots & \vdots & 0 & \ddots & \ddots & 0 & \vdots & \vdots & \vdots \\
	\l(\gfrac{n-3}{n-2}\r) & 0 & 0 & 0 & * & 0 & 0 & 0 & 0 & 0 & 0 & 0 & 0 & 0 & 0 \\
	\cline{1-15}
    	\l(\gfrac{n-2}{n-1}\r) & 0 & 0 & 0 & 0 & * & 0 & * & 0 & 0 & 0 & 0 & 0 & 0 & * \\
	\l(\gfrac{n-2}{n}\r) & 0 & 0 & 0 & 0 & * & * 	 & 0 & 0 & 0 & 0 & 0 & 0 & 0 & * \\
	\l(\gfrac{n-2}{n-1 \: \: n}\r) & 0 & 0 & 0 & 0 & * & 0 & 0 & 0 & 0 & 0 & 0 & * & * & 0 \\
    	\end{block}
  	\end{blockarray}
	\]
	
	
	
	Recall the results of \cref{lem:rad,lem:ext1projnohom,lem:sinknohom}:
	\begin{align}
	&\Ext^1_{D(n)}(\tilde{P_i},\tilde{S_j})=0=\Ext^1_{D(n)}(\tilde{S_i},\tilde{I_j}) \text{ if }i\not=j \label{eq:simproj}\\
	&\Ext^1_{D(n)}(\tilde{P_i},\tilde{P_j})=0=\Ext^1_{D(n)}(\tilde{I_i},\tilde{I_j}) \text{ if they are not simple} \label{eq:radproj} \\
	&\Ext^1_{D(n)}(\tilde{P_i},M)=0 \text{ if }\Hom_{D(n)}(P_i,M)=0 \\
	&\Ext^1_{D(n)}(\tilde{S_i},M)=0 \text{ if }i\text{ is a sink in }Q'\text{ and }\Hom_{D(n)}(P_i,M)=0.
	\end{align}
	
	By \cref{eq:simproj} and \cref{eq:radproj}, we get all of the entries in the four top left block matrices. We also get the entries in the middle rightmost block, the middle bottom block, and the bottom rightmost block. It remains to check the top rightmost block and the bottom leftmost block. 
	
	For the top rightmost block, note that \cref{eq:radproj} gives all of the entries in the $\pgfrac{n-2}{n-1}, \pgfrac{n-2}{n}$ columns. 
	
	\vsp
	
	{\bf Claim.}
	$\Ext^1_{D(n)}\l(\l(i\r), \pgfrac{n-2}{n-1 \:\: n}\r)=0$ for $i \leq n-4.$
	
	\vsp
	
	We have the following projective resolution for $(i)=S_i$:
	\[
	\begin{tikzcd}
	\cdots \arrow[r] & P_{i+1} \oplus P_i^{N_i} \arrow[r] \arrow[rd] & P_i \arrow[r] & S_i \arrow[r] & 0 \\
	{} & {} & \kk x_{i+1} \oplus \l( \bigoplus_{l=1}^{N_i} \kk a_i^l \r) \arrow[hookrightarrow]{u} & {} & {}
	\end{tikzcd}
	\]
	
	Now take $\Hom_{D(n)}\l(-,\pgfrac{n-2}{n-1\:\: n}\r)=\Hom_{D(n)}(-,\tilde{P}_{n-2})$:
	\[
	\begin{tikzcd}
	0 \arrow[r] & \Hom_{D(n)}(P_i,\tilde{P}_{n-2}) \arrow[r] & \Hom_{D(n)}(P_{i+1},\tilde{P}_{n-2}) \oplus \Hom_{D(n)}(P_i,\tilde{P}_{n-2})^{N_i} \arrow[r] & \cdots
	\end{tikzcd}
	\]
	
	Then for $j\not = n-2$,
	\begin{align*}
	\dim \Hom_{D(n)}(P_j, \tilde{P}_{n-2})&=\dim \Hom_{D(n)}\l(D(n)e_j, D(n)e_{n-2}/Je_{n-2}\r) \\
	&=\dim e_jD(n)e_{n-2}/Je_{n-2} \\
	&=\text{ the number of paths }n-2 \to j.
	\end{align*}
	Since $i \leq n-4$, there are no paths from $n-2$ to $j=i$ or to $j=i+1$. Therefore, $\Hom_{D(n)}(P_{i+1},\tilde{P}_{n-2}) \oplus \Hom_{D(n)}(P_i,\tilde{P}_{n-2})^{N_i}=0$.
	
	Now let us examine the bottom left block matrix. We already know the $\pgfrac{n-2}{n-1\:\:n}$ row of the matrix by \cref{eq:simproj}. 
	
	\vsp
	
	{\bf Claim.}
	$\Ext^1_{D(n)}\l(\pgfrac{n-2}{n},\l(i\r)\r)=0$ for $i\not\in\{ n-1, n-2\}$, and $\Ext^1_{D(n)}\l(\pgfrac{n-2}{n-1},\l(i\r)\r)=0$ for $i\not\in\{ n, n-2\}$.
	
	\vsp

	We have the following projective resolution of $\pgfrac{n-2}{n},$ which is equivalent as a $\kk$-vector space to $\frac{P_{n-2}}{\kk x_{n-1} \oplus \l(\bigoplus_{l=1}^{N_{n-2}} \kk a^l_{n-2}\r)}$:
	
	\[
	\begin{tikzcd}
	\cdots \arrow[r] & P_{n-1} \oplus P_{n-2}^{N_{n-2}}\arrow[r]\arrow[dr] & P_{n-2} \arrow[r]& \pgfrac{n-2}{n}\arrow[r] & 0 \\
	{} & {} & kx_{n-1}  \oplus\l(\bigoplus_{l=1}^{N_{n-2}} ka^l_{n-2}\r) \arrow[hookrightarrow]{u} & {} & {}
	\end{tikzcd}
	\]
	
	Take $\Hom_{D(n)}(-, (i))=\Hom_{D(n)}(-, S_i)$:
	\[
	\begin{tikzcd}
	0 \arrow[r] & \Hom_{D(n)}(P_{n-2},S_i)\arrow[r] & \Hom_{D(n)}(P_{n-1},S_i) \oplus \Hom_{D(n)}(P_{n-2},S_i)^{N_{n-2}}\arrow[r] & \cdots
	\end{tikzcd}
	\]
	
	Since $n-1\not=i\not=n-2$, $\Hom_{D(n)}(P_{n-1},S_i) \oplus \Hom_{D(n)}(P_{n-2},S_i)^{N_{n-2}}=0$ so $\Ext^1_{D(n)}\l(\pgfrac{n-2}{n},\l(i\r)\r)=0$.
	
	The proof for $\pgfrac{n-2}{n-1}$ is analogous.
		
	\vsp
	
	{\bf Claim.}
	Let $\phi$ be any brick set. Then $\rho(A(\varphi)) \leq \max\{N_1, \ldots, N_n\}.$ Since $\phi=\{(1), \ldots, (n)\}$ is a brick set such that $\rho(A(\varphi)) =\max\{N_1, \ldots, N_n\},$ we must have $\fpd \l(D(n)\dash\bmod\r) = \max\{N_1, \ldots, N_n\}$.
	
	\vsp
	
	First notice that $\l\{\pgfrac{n-2}{n}, \pgfrac{n-2}{n-1}\r\}$ and $\l\{\pgfrac{n-2}{n-1 \: \: n}\r\}$ are maximal brick subsets of $\l\{\pgfrac{n-2}{n}, \pgfrac{n-2}{n-1}, \pgfrac{n-2}{n-1 \: \: n}\r\}$. Let $\phi$ be a brick set. We can divide our proof of the claim into the following cases.
	
	\vsp
	
	\emph{Case 1: $\pgfrac{n-2}{n}, \pgfrac{n-2}{n-1}, \pgfrac{n-2}{n-1 \: \: n} \not \in \phi.$}
	
	Then for some $1 \leq j_l \leq n-3, \: 1 \leq i_m \leq n$,
	$$\phi=\l\{ (i_1), \ldots, (i_M), \pgfrac{j_1}{j_1+1}, \ldots, \pgfrac{j_L}{j_L+1} \r\}.$$ 
	
	Note that by examination of the $\Hom$ matrix, if $(i),\pgfrac{j}{j+1}$ are in the same brick set, then $j \not = i \not = j+1$, so
	\begin{equation}
	\label{eq:ijjplus1}
	\Ext^1_{D(n)}\l((i),\pgfrac{j}{j+1}\r)=0=\Ext^1_{D(n)}\l(\pgfrac{j}{j+1}, (i)\r).
	\end{equation} 
	
	Then by examination of the $\Ext^1$ matrix we have for some $* \in \{0,1\}$:
	\[
	A(\phi)=
	\begin{blockarray}{ccccccc|ccccc}
    	{} & (i_1) & (i_2) & (i_3) & \cdots & (i_{M-1}) & (i_M) & \l(\gfrac{j_1}{j_1+1}\r) & \l(\gfrac{j_2}{j_2+1}\r) & \cdots & \l(\gfrac{j_{L-1}}{j_{L-1}+1}\r) & \l(\gfrac{j_L}{j_L+1}\r) \\
    	\begin{block}{c(cccccc|ccccc@{\hspace*{5pt}})}
    	(i_1) & N_{i_1} & * & 0 & 0 & \cdots & 0 						& 0 & 0 & 0 & 0 & 0 \\
    	(i_2) & 0 & N_{i_2} & * & 0 & \cdots & 0 						& 0 & 0 & 0 & 0 & 0 \\
		(i_3) & 0 & 0 & N_{i_3} & * & \cdots & 0							& 0 & 0 & 0 & 0 & 0 \\
    	\vdots&\vdots&\vdots&\ddots&\ddots&\ddots&\vdots		& 0 & 0 & 0 & 0 & 0 \\
		(i_{M-1}) & 0 & 0  & 0 & 0 & N_{i_{M-1}} & * 				& 0 & 0 & 0 & 0 & 0 \\
    	(i_M) &  0 & 0 & 0 & 0 & 0 & N_{i_M}								& 0 & 0 & 0 & 0 & 0 \\
    	\cline{1-12}
    	\l(\gfrac{j_1}{j_1+1}\r) & 0 & 0 & 0 & 0 & 0 & 0			& 0 & 0 & 0 & 0 & 0 \\
    	\l(\gfrac{j_2}{j_2+2}\r) & 0 & 0 & 0 & 0 & 0 & 0			& 0 & 0 & 0 & 0 & 0 \\
    	\vdots & 0 & 0 & 0 & 0 & 0 & 0 										& 0 & 0 & 0 & 0 & 0 \\
		\l(\gfrac{j_{L-1}}{j_{L-1}+1}\r) & 0 & 0 & 0 & 0 & 0 & 0 	& 0 & 0 & 0 & 0 & 0 \\
    	\l(\gfrac{j_L}{j_L+1}\r) & 0 & 0 & 0 & 0 & 0 & 0 	 		& 0 & 0 & 0 & 0 & 0 \\
    	\end{block}
  	\end{blockarray}
	\]

	If there are no simples,
	\[
	\rho(A(\phi))=0.
	\]
	Else, 
	\[
	\rho(A(\phi))=\max\{N_{i_1},\ldots, N_{i_M}\} \leq \max\{N_1, \ldots, N_n\}.
	\]
	
	\vsp
	
	\emph{Case 2: $\pgfrac{n-2}{n}\in \phi, \pgfrac{n-2}{n-1}\not\in \phi.$} 
	
	Then for some $1 \leq j_l \leq n-3, \: 1 \leq i_m \leq n$,
	$$\phi=\l\{ (i_1), \ldots, (i_M), \pgfrac{j_1}{j_1+1}, \ldots, \pgfrac{j_L}{j_L+1}, \pgfrac{n-2}{n} \r\}.$$ 
	
	Suppose that $i_1 \leq \cdots \leq i_M.$ By examination of the $\Hom$ matrix, since $\pgfrac{n-2}{n} \in \phi$ we have $\l(n-2\r),\l(n\r), \pgfrac{n-2}{n-1 \: \: n} \not \in \phi.$ Since $(n) \not \in \phi$, the largest possible value of $i_M$ is $n-1$.
		
	Let $*$ denote a possible non-zero entry. We will show that 
	\[
	A(\phi)=
	\begin{blockarray}{ccccccc|c|ccccc}
    	{} & (i_1) & (i_2) & (i_3) & \cdots & (i_{M-1}) & (i_M) & \pgfrac{n-2}{n} & \l(\gfrac{j_1}{j_1+1}\r) & \l(\gfrac{j_2}{j_2+1}\r) & \cdots & \l(\gfrac{j_{L-1}}{j_{L-1}+1}\r) & \l(\gfrac{j_L}{j_L+1}\r) \\
    	\begin{block}{c(cccccc|c|ccccc@{\hspace*{5pt}})}
    	(i_1) & N_{i_1} & * & 0 & 0 & \cdots & 0 						& 0 & 0 & 0 & 0 & 0 & 0 \\
    	(i_2) & 0 & N_{i_2} & * & 0 & \cdots & 0 						& 0 & 0 & 0 & 0 & 0 & 0 \\
		(i_3) & 0 & 0 & N_{i_3} & * & \cdots & 0							& 0 & 0 & 0 & 0 & 0 & 0 \\
    	\vdots&\vdots&\vdots&\ddots&\ddots&\ddots&0				& 0 & 0 & 0 & 0 & 0 & 0 \\
		(i_{M-1}) & 0 & 0  & 0 & 0 & N_{i_{M-1}} & * 				& 0 & 0 & 0 & 0 & 0 & 0 \\
    	(i_M) &  0 & 0 & 0 & 0 & 0 & N_{i_M}								& 0 & 0 & 0 & 0 & 0 & 0 \\
    	\cline{1-13}
    	\pgfrac{n-2}{n} &  0 & 0 & 0 & 0 & 0 & *						& 0 & 0 & 0 & 0 & 0 & 0 \\
    	\cline{1-13}
    	\l(\gfrac{j_1}{j_1+1}\r) & 0 & 0 & 0 & 0 & 0 & 0			& 0 & 0 & 0 & 0 & 0 & 0 \\
    	\l(\gfrac{j_2}{j_2+2}\r) & 0 & 0 & 0 & 0 & 0 & 0			& 0 & 0 & 0 & 0 & 0 & 0 \\
    	\vdots & 0 & 0 & 0 & 0 & 0 & 0 										& 0 & 0 & 0 & 0 & 0 & 0 \\
		\l(\gfrac{j_{L-1}}{j_{L-1}+1}\r) & 0 & 0 & 0 & 0 & 0 & 0 	& 0 & 0 & 0 & 0 & 0 & 0 \\
    	\l(\gfrac{j_L}{j_L+1}\r) & 0 & 0 & 0 & 0 & 0 & 0 	 		& 0 & 0 & 0 & 0 & 0 & 0 \\
    	\end{block}
  	\end{blockarray},
	\]
	
	We have the top right block matrix because $(n-2) \not \in \phi$. We have the middle row and column since $(n-2),(n) \not \in \phi$. By \cref{eq:ijjplus1}, we have the rest of the $A(\phi)$ matrix.
	
	If there are no simples,
	\[
	\rho(A(\phi))=0.\] 
	Else,
	\[
	\rho(A(\phi))=\max\{N_{i_1}, \ldots, N_{i_M}\} \leq \max\{N_1, \ldots, N_n\}.\]
		
	\vsp

	\emph{Case 3: $\pgfrac{n-2}{n-1}\in \phi, \pgfrac{n-2}{n}\not\in\phi.$}
	
	By symmetry, we can use the proof for Case 2.
	
	\vsp
	
	\emph{Case 4: $\pgfrac{n-2}{n-1},\pgfrac{n-2}{n}\in\phi.$}
	
	Then for some $1 \leq j_l \leq n-3, \: 1 \leq i_m \leq n$,
	$$\phi=\l\{ (i_1), \ldots, (i_M), \pgfrac{j_1}{j_1+1}, \ldots, \pgfrac{j_L}{j_L+1}, \pgfrac{n-2}{n-1},\pgfrac{n-2}{n} \r\}.$$ 
		
	Because $\pgfrac{n-2}{n-1},\pgfrac{n-2}{n}\in\phi,$ it is clear from the $\Hom$ matrix that $(n-2),(n-1),(n),  \pgfrac{n-2}{n-1 \: \: n} \not \in \phi$. Therefore we have the $A(\phi)$ matrix below:
	\[
	\begin{blockarray}{ccccccc|cc|ccccc}
    	{} & (i_1) & (i_2) & (i_3) & \cdots & (i_{M-1}) & (i_M) & \pgfrac{n-2}{n-1} &\pgfrac{n-2}{n} & \l(\gfrac{j_1}{j_1+1}\r) & \l(\gfrac{j_2}{j_2+1}\r) & \cdots & \l(\gfrac{j_{L-1}}{j_{L-1}+1}\r) & \l(\gfrac{j_L}{j_L+1}\r) \\
    	\begin{block}{c(cccccc|cc|ccccc@{\hspace*{5pt}})}
    	(i_1) & N_{i_1} & * & 0 & 0 & \cdots & 0 						& 0 & 0 & 0 & 0 & 0 & 0 & 0 \\
    	(i_2) & 0 & N_{i_2} & * & 0 & \cdots & 0 						& 0 & 0 & 0 & 0 & 0 & 0 & 0\\
		(i_3) & 0 & 0 & N_{i_3} & * & \cdots & 0							& 0 & 0 & 0 & 0 & 0 & 0 & 0\\
    	\vdots&\vdots&\vdots&\ddots&\ddots&\ddots&0				& 0 & 0 & 0 & 0 & 0 & 0& 0 \\
		(i_{M-1}) & 0 & 0  & 0 & 0 & N_{i_{M-1}} & * 				& 0 & 0 & 0 & 0 & 0 & 0& 0 \\
    	(i_M) &  0 & 0 & 0 & 0 & 0 & N_{i_M}								& 0 & 0 & 0 & 0 & 0 & 0& 0 \\
    	\cline{1-14}
    	\pgfrac{n-2}{n-1} &  0 & 0 & 0 & 0 & 0 & 0					& 0 & 0 & 0 & 0 & 0 & 0& 0 \\
    	\pgfrac{n-2}{n} &  0 & 0 & 0 & 0 & 0 & 0						& 0 & 0 & 0 & 0 & 0 & 0& 0 \\
    	\cline{1-14}
    	\l(\gfrac{j_1}{j_1+1}\r) & 0 & 0 & 0 & 0 & 0 & 0			& 0 & 0 & 0 & 0 & 0 & 0& 0 \\
    	\l(\gfrac{j_2}{j_2+2}\r) & 0 & 0 & 0 & 0 & 0 & 0			& 0 & 0 & 0 & 0 & 0 & 0& 0 \\
    	\vdots & 0 & 0 & 0 & 0 & 0 & 0 										& 0 & 0 & 0 & 0 & 0 & 0& 0 \\
		\l(\gfrac{j_{L-1}}{j_{L-1}+1}\r) & 0 & 0 & 0 & 0 & 0 & 0 & 0 & 0 & 0 & 0 & 0 & 0& 0 \\
    	\l(\gfrac{j_L}{j_L+1}\r) & 0 & 0 & 0 & 0 & 0 & 0 	 		& 0 & 0 & 0 & 0 & 0 & 0 & 0\\
    	\end{block}
  	\end{blockarray}
	\]
	
	If there are no simples,
	\[
	\rho(A(\phi))=0.
	\]
	Else,
	\[
	\rho(A(\phi))=\max\{N_{i_1}, \ldots, N_{i_M}\} \leq \max\{N_1, \ldots, N_n\}.
	\]
	
	\vsp
	
	\emph{Case 5: $\pgfrac{n-2}{n-1\:\:n} \in \phi$.}
	
	Then for some $1 \leq j_l \leq n-3, \: 1 \leq i_m \leq n$,
	$$\phi=\l\{ (i_1), \ldots, (i_M), \pgfrac{j_1}{j_1+1}, \ldots, \pgfrac{j_L}{j_L+1}, \pgfrac{n-2}{n-1\:\:n} \r\}.$$ 	
	
	By examination of the $\Hom$ matrix, since $\pgfrac{n-2}{n-1\:\:n} \in \phi$, we cannot have $(n-2),(n-1),(n), \pgfrac{n-2}{n-1}$ or $\pgfrac{n-2}{n}$ in $\phi$. 
	
	\[
	A(\phi)=
	\begin{blockarray}{ccccccc|c|ccccc}
    	{} & (i_1) & (i_2) & (i_3) & \cdots & (i_{M-1}) & (i_M) & \pgfrac{n-2}{n-1\:\:n} & \l(\gfrac{j_1}{j_1+1}\r) & \l(\gfrac{j_2}{j_2+1}\r) & \cdots & \l(\gfrac{j_{L-1}}{j_{L-1}+1}\r) & \l(\gfrac{j_L}{j_L+1}\r) \\
    	\begin{block}{c(cccccc|c|ccccc@{\hspace*{5pt}})}
    	(i_1) & N_{i_1} & * & 0 & 0 & \cdots & 0 						& 0 & 0 & 0 & 0 & 0 & 0 \\
    	(i_2) & 0 & N_{i_2} & * & 0 & \cdots & 0 						& 0 & 0 & 0 & 0 & 0 & 0 \\
		(i_3) & 0 & 0 & N_{i_3} & * & \cdots & 0							& 0 & 0 & 0 & 0 & 0 & 0 \\
    	\vdots&\vdots&\vdots&\ddots&\ddots&\ddots&0				& 0 & 0 & 0 & 0 & 0 & 0 \\
		(i_{M-1}) & 0 & 0  & 0 & 0 & N_{i_{M-1}} & * 				& 0 & 0 & 0 & 0 & 0 & 0 \\
    	(i_M) &  0 & 0 & 0 & 0 & 0 & N_{i_M}								& * & 0 & 0 & 0 & 0 & 0 \\
    	\cline{1-13}
    	\pgfrac{n-2}{n-1\:\:n} &  0 & 0 & 0 & 0 & 0 & 0						& 0 & 0 & 0 & 0 & 0 & 0 \\
    	\cline{1-13}
    	\l(\gfrac{j_1}{j_1+1}\r) & 0 & 0 & 0 & 0 & 0 & 0			& 0 & 0 & 0 & 0 & 0 & 0 \\
    	\l(\gfrac{j_2}{j_2+2}\r) & 0 & 0 & 0 & 0 & 0 & 0			& 0 & 0 & 0 & 0 & 0 & 0 \\
    	\vdots & 0 & 0 & 0 & 0 & 0 & 0 										& 0 & 0 & 0 & 0 & 0 & 0 \\
		\l(\gfrac{j_{L-1}}{j_{L-1}+1}\r) & 0 & 0 & 0 & 0 & 0 & 0 	& 0 & 0 & 0 & 0 & 0 & 0 \\
    	\l(\gfrac{j_L}{j_L+1}\r) & 0 & 0 & 0 & 0 & 0 & 0 	 		& 0 & 0 & 0 & 0 & 0 & 0 \\
    	\end{block}
  	\end{blockarray}
	\]
	
	This can be checked by examination of the $\Ext^1$ matrix. If there are no simples,
	\[
	\rho(A(\phi))=0.
	\]
	Else,
	\[
	\rho(A(\phi))=\max\{N_{i_1}, \ldots, N_{i_M}\} \leq \max\{N_1, \ldots, N_n\}.
	\]
	
	\vsp
	
	We have exhausted all cases (and possibly the reader). 
	\end{proof}
	
	\vsp
	

	\subsection{$E(6), E(7), E(8)$ algebras}
	
	\begin{theorem} \label{thm:En1dir}
	Suppose that $n \in \{6,7,8\}$, and let $E(n)$ be the finite dimensional radical square zero algebra described by the following quiver $Q$ with $N_i$ loops (labeled $a_i^l$ for $l=1, \ldots, N_i$) at vertex i:
	
	\[
	\begin{tikzcd}
	{} & {} & 2 \arrow[out=65, in=115, loop, swap, "a_{2}^l"] & {}  \\
	1   \arrow[r, "x_3"]  \arrow[out=252, in=-72, loop] \arrow[out=245, in=-65, loop] \arrow[out=238, in=-58, loop, swap, "a_1^l"] & 3 \arrow[out=245, in=-65, loop, swap, "a_3^l"] \arrow[r, "x_4"] & 4 \arrow[out=252, in=-72, loop] \arrow[out=245, in=-65, loop] \arrow[out=238, in=-58, loop, swap, "a_{4}^l"] \arrow[r, "x_5"] \arrow[u, "x_{2}"] & \cdots \arrow[r, "x_n"] & n  \arrow[out=245, in=-65, loop] \arrow[out=238, in=-58, loop, swap, "a_n^l"]
	\end{tikzcd} 
	\]
	
	Then
	\[
	\fpd \l(E(n)\dash\bmod\r) = \max\{N_1, \ldots, N_n\}.
	\]
	\end{theorem}
		
	\vsp
	
	\begin{proof}
	
	We will reuse some notation from the previous theorem. We have $E(n)=\kk Q/(\geq 2).$ Let $J$ be the ideal generated by the loops, and define $B:=E(n)/J.$ In other words, $B$ is the radical square zero algebra described by the following quiver $Q'$:
	
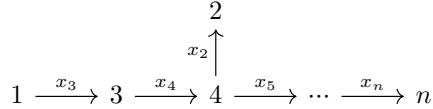
	\[
	\begin{tikzcd}
	{} & {} & 2 & {}  \\
	1   \arrow[r, "x_3"]  & 3  \arrow[r, "x_4"] & 4   \arrow[r, "x_5"] \arrow[u, "x_{2}"] & \cdots \arrow[r, "x_n"] & n	
	\end{tikzcd} 
	\]	
	
	We can see that $B=C/(\geq 2),$ where $C:=\kk Q'$.
	
	By \cref{lem:quivpowerful}, we have a one-to-one correspondence between bricks of $E(n)$ and bricks of $C$ annihilated by $(\geq 2)$, as well as the following natural isomorphism:
	\[
	\Hom_{E(n)}(M,N) \cong \Hom_{B}(M,N) \cong \Hom_{C}(M,N).
	\]
	Hence it suffices to compute the dimensions of $\Hom_{E(n)}$-spaces in $B$ or $C$. 
	
	The indecomposable modules of $C$ are well known by Gabriel's theorem \cite{Gabriel}. The dimension vectors of these modules correspond precisely to the positive roots for the $E_n$ root system, which are described in \cite[p. 64]{Humphreys}. An explicit description of the positive roots for $E_6, E_7, E_8$ is given in \cite{E6roots,E7roots,E8roots}.
	
	We will show that the only indecomposable $C$-modules annihilated by $(\geq 2)$ are:
	\begin{equation} \label{eq:Eninds}
	\l\{\l(1\r), \l(2\r), \ldots, \l(n\r), \pgfrac{1}{3}, \pgfrac{3}{4}, \pgfrac{5}{6}, \ldots,\pgfrac{n-1}{n}, \pgfrac{4}{2}, \pgfrac{4}{5}, \pgfrac{4}{2 \:\: 5}\r\}.
	\end{equation}
	
	This follows from examining the positive roots for $E_n$ with respect to the next proposition.
	
		\begin{proposition} \label{prp:Eninds}
		The following indecomposable representations of $C$ are not annihilated by $(\geq 2)$ for any $l,m,k \geq 1$ for any entries $*$:

		\begin{align}
		& (l, *, m,k, *, \ldots, *) \label{eq:25} \\
		& (*,*,*,l,m,k,*, \ldots, *) \\
		& (*,*,*,l,m,k) \text{ for }N=6\\
		& (*,*,*,*,l,m,k) \text{ for }N=7 \\
		& (*,*,*,*,*,l,m,k) \text{ for }N=8 \\
		& (*,*,*,*,l,m,k, *) \text{ for }N=8 \label{eq:30} \\
		& (*,1,1,2,1,*, \ldots, *) \label{eq:31} \\
		& (*,1,1,1,*, \ldots, *) \label{eq:32} \\
		& (*,*, 1,1,1, *, \ldots, *). \label{eq:33}
		\end{align}
		\end{proposition}
	
		\begin{proof}
		Now in the case of $m=1$, \cref{eq:25} through \cref{eq:30} follow from \cref{fct:collection} \cref{prp:notannihgeqn} and \cref{prp:indreps}, since all of these representations have a path $\beta \alpha$ of length $2$ such that $f_{\beta} \circ f_{\alpha} \not =0$. We also have \cref{eq:32,eq:33} for the same reason. By \cref{prp:notindifannih}, we have \cref{eq:25} through \cref{eq:30} for $m\geq 2$. 
		
		The proof of \cref{eq:31} is almost identical to the proof of \cref{lem:notDnindreps} \cref{eq:notDnindreps2}, so we omit it.
		\end{proof}
	
	The indecomposables given in \cref{eq:Eninds} are bricks of $E(n)$ by examination of the following $\Hom$ matrix, where $H_{ij}=\dim \Hom_{E(n)}(X_i,X_j)$ for some brick objects $X_i, X_j$ of $E(n)\dash\bmod$:
	
	\[
  	\begin{blockarray}{ccccccccc|ccccc|cc|c}
    	{} & (1) & (2) & (3) & (4) & (5) & \cdots & (n-1) & (n) & \l(\gfrac{1}{3}\r) & \l(\gfrac{3}{4}\r) & \pgfrac{5}{6} & \cdots & \pgfrac{n-1}{n} & \l(\gfrac{4}{2}\r) & \pgfrac{4}{5} & \pgfrac{4}{2 \: \: 5} \\
    	\begin{block}{c(cccccccc|ccccc|cc|c@{\hspace*{5pt}})}
    	(1) & 1 & 0 & 0 & 0 & 0 & 0 & 0 & 0 & 0 & 0 & 0 & 0 & 0 & 0 & 0 & 0 \\
    	(2) & 0 & 1 & 0 & 0 & 0 & 0 & 0 & 0 & 0 & 0 & 0 & 0 & 0 & 1 & 0 & 1 \\
	(3) & 0 & 0 & 1 & 0 & 0 & 0 & 0 & 0 & 1 & 0 & 0 & 0 & 0 & 0 & 0 & 0 \\
	(4) & 0 & 0 & 0 & 1 & 0 & 0 & 0 & 0 & 0 & 1 & 0 & 0 & 0 & 0 & 0 & 0\\
    	(5) & 0 & 0 & 0 & 0 & 1 & \ddots & 0 & 0 & 0 & 0 & 0 & 0 & 0 & 0 & 1 & 1\\
	\vdots & 0 & 0  & 0 & 0 & 0 & \ddots & \ddots & 0 & 0 & 0 & 1 & 0 & 0 & 0 & 0 & 0 \\
	\vdots & 0 & 0  & 0 & 0 & 0 & 0 & \ddots & 0 & 0 & 0 & 0 & \ddots & 0 & 0 & 0 & 0 \\
	(n) &  0 & 0 & 0 & 0 & 0 & 0 & 0 & 1 & 0 & 0 & 0 & 0 & 1 & 0 & 0 & 0 \\
    	\cline{1-17}
    	\pgfrac{1}{3} & 1 & 0 & 0 & 0 & 0 & 0 & 0 & 0 & 1 & 0 & 0 & 0 & 0 & 0 & 0 & 0 \\
    	\l(\gfrac{3}{4}\r) & 0 & 0 & 1 & 0 & 0 & 0 & 0 & 0 & 1 & 1 & 0 & 0 & 0 & 0 & 0 & 0\\
    	\pgfrac{5}{6} & 0 & 0 & 0 & 0 & 1 & 0 & 0 & 0 & 0 & 0 & 1 & 0 & 0 & 0 & 1 & 1 \\
	\vdots & 0 & 0 & 0 & 0 & 0 & \ddots & 0 & 0 & 0 & 0 & 1 & \ddots & 0 & 0 & 0 & 0\\
	\pgfrac{n-1}{n} & 0 & 0 & 0 & 0 & 0 & 0 & 1 & 0 & 0 & 0 & 0 & 1 & 1 & 0 & 0 & 0\\
	\cline{1-17}
    	\l(\gfrac{4}{2}\r) & 0 & 0 & 0 & 1 & 0 & 0 & 0 & 0 & 0 & 1 & 0 & 0 & 0 & 1 & 0 & 0\\
	\l(\gfrac{4}{5}\r) & 0 & 0 & 0 & 1 & 0 & 0 & 0 & 0 & 0 & 1 & 0 & 0 & 0 & 0 & 1 & 0\\
	\cline{1-17}
	\l(\gfrac{4}{2 \: \: 5}\r) & 0 & 0 & 0 & 1 & 0 & 0 & 0 & 0 & 0 & 1 & 0 & 0 & 0 & 1 & 1 & 1\\
    	\end{block}
  	\end{blockarray}
	\]
	
	To calculate the entries of the $\Hom$ matrix, recall from \cref{fct:collection} \cref{prp:dimhoms1} and \cref{prp:dimhoms2} that
	\begin{align*}
	\dim \Hom_{C}(\tilde{P_i},\tilde{S_j})&= \delta_{ij}, \\
	\dim \Hom_{C}(\tilde{S_i},\tilde{I_j})&=\delta_{ij}, \\
	\dim \Hom_{C}(\tilde{I_j},\tilde{I_i})&= \text{ the number of paths from }i\rightarrow j \text{ in }C = \dim \Hom_{C}(\tilde{P_j},\tilde{P_i}).
	\end{align*}

	Furthermore, if there is no overlap between the two representations (that is, one representation has non-zero vector spaces only at vertices $i_1, \ldots, i_M$ and the other representation has non-zero vector spaces only at vertices $j_1, \ldots, j_L$ such that $i_m \not= j_l$ for any $m,l$), then there are no morphisms between them. The remaining $\Hom$ spaces can be calculated:
	
	\begin{align*}
	\dim\Hom_{C}\l( \l(4\r), \pgfrac{4}{2\:\: 5}\r)=0 \\
	\dim\Hom_{C}\l( \l(5\r), \pgfrac{4}{2\:\: 5}\r)=1 \\
	\dim\Hom_{C}\l( \pgfrac{4}{2},\l(2\r) \r)=0 \\
	\dim\Hom_{C}\l( \pgfrac{4}{2},\l(4\r)\r)=1 \\
	\dim\Hom_{C}\l( \pgfrac{4}{5},\l(4\r)\r)=1 \\
	\dim\Hom_{C}\l(\pgfrac{4}{5},\l(5\r) \r)= 0 \\
	\dim\Hom_{C}\l( \pgfrac{4}{2}, \pgfrac{4}{2\:\: 5}\r)=0 \\
	\dim\Hom_{C}\l( \pgfrac{4}{5}, \pgfrac{4}{2\:\: 5}\r)=0 \\
	\dim\Hom_{C}\l( \pgfrac{4}{2\:\: 5},\pgfrac{4}{2} \r)= 1\\
	\dim\Hom_{C}\l( \pgfrac{4}{2\:\: 5},\pgfrac{4}{5} \r)= 1.\\
	\end{align*}
	
	We have the following $\Ext^1$ matrix, where using the same notation as before, $E_{ij}=\dim \Ext^1_{E(n)}(X_i,X_j),$ and $*$ denotes a possible non-zero entry:
	
	\[
  	\begin{blockarray}{ccccccccc|ccccc|cc|c}
    	{} & (1) & (2) & (3) & (4) & (5) & \cdots & (n-1) & (n) & \l(\gfrac{1}{3}\r) & \l(\gfrac{3}{4}\r) & \pgfrac{5}{6} & \cdots & \pgfrac{n-1}{n} & \l(\gfrac{4}{2}\r) & \pgfrac{4}{5} & \pgfrac{4}{2 \: \: 5} \\
    	\begin{block}{c(cccccccc|ccccc|cc|c@{\hspace*{5pt}})}
    	(1) & N_1 & 0 & 1 & 0 & 0 & 0 & 0 & 0 & 0 & 0 & 0 & 0 & 0 & 0 & 0 & * \\
    	(2) & 0 & N_2 & 0 & 0 & 0 & 0 & 0 & 0 & 0 & 0 & 0 & 0 & 0 & * & 0 & * \\
	(3) & 0 & 0 & N_3 & 1 & 0 & 0 & 0 & 0 & * & 0 & 0 & 0 & 0 & 0 & 0 & * \\
	(4) & 0 & 1 & 0 & N_4 & 1 & 0 & 0 & 0 & 0 & * & 0 & 0 & 0 & 0 & 0 & *\\
    	(5) & 0 & 0 & 0 & 0 & N_5 & \ddots & 0 & 0 & 0 & 0 & 0 & 0 & 0 & 0 & * & *\\
	\vdots & 0 & 0  & 0 & 0 & 0 & \ddots & \ddots & 0 & 0 & 0 & * & 0 & 0 & 0 & 0 & * \\
	\vdots & 0 & 0  & 0 & 0 & 0 & 0 & \ddots & 1 & 0 & 0 & 0 & \ddots & 0 & 0 & 0 & * \\
	(n) &  0 & 0 & 0 & 0 & 0 & 0 & 0 & N_n & 0 & 0 & 0 & 0 & * & 0 & 0 & *\\
    	\cline{1-17}
    	\pgfrac{1}{3} & * & 0 & 0 & 0 & 0 & 0 & 0 & 0 & 0 & 0 & 0 & 0 & 0 & 0 & 0 & 0 \\
    	\l(\gfrac{3}{4}\r) & 0 & 0 & * & 0 & 0 & 0 & 0 & 0 & 0 & 0 & 0 & 0 & 0 & 0 & 0 & 0\\
    	\pgfrac{5}{6} & 0 & 0 & 0 & 0 & * & 0 & 0 & 0 & 0 & 0 & 0 & 0 & 0 & 0 & 0 & 0 \\
	\vdots & 0 & 0 & 0 & 0 & 0 & \ddots & 0 & 0 & 0 & 0 & 0 & 0 & 0 & 0 & 0 & 0\\
	\pgfrac{n-1}{n} & 0 & 0 & 0 & 0 & 0 & 0 & * & 0 & 0 & 0 & 0 & 0 & 0 & 0 & 0 & 0\\
	\cline{1-17}
    	\l(\gfrac{4}{2}\r) & 0 & 0 & 0 & * & * & 0 & 0 & 0 & 0 & 0 & 0 & 0 & 0 & 0 & 0 & *\\
	\l(\gfrac{4}{5}\r) & 0 & * & 0 & * & 0 & 0 & 0 & 0 & 0 & 0 & 0 & 0 & 0 & 0 & 0 & *\\
	\cline{1-17}
	\l(\gfrac{4}{2 \: \: 5}\r) & 0 & 0 & 0 & * & 0 & 0 & 0 & 0 & 0 & 0 & 0 & 0 & 0 & * & * & 0\\
    	\end{block}
  	\end{blockarray}
	\]
	
	Recall the results of \cref{lem:rad,lem:ext1projnohom,lem:sinknohom}:
	\begin{align*}
	&\Ext^1_{E(n)}(\tilde{P_i},\tilde{S_j})=0=\Ext^1_{E(n)}(\tilde{S_i},\tilde{I_j}) \text{ if }i\not=j \\
	&\Ext^1_{E(n)}(\tilde{P_i},\tilde{P_j})=0=\Ext^1_{E(n)}(\tilde{I_i},\tilde{I_j}) \text{ if they are not simple} \\
	&\Ext^1_{E(n)}(\tilde{P_i},M)=0 \text{ if }\Hom_{E(n)}(P_i,M)=0 \\
	&\Ext^1_{E(n)}(\tilde{S_i},M)=0 \text{ if }i\text{ is a sink in }Q'\text{ and }\Hom_{E(n)}(P_i,M)=0 \\
	&\dim \Ext^1_{E(n)}((i),(j))= \text{ the number of arrows }i \rightarrow j.
	\end{align*}
	
	These equations give us all the zero entries of the $\Ext^1$ matrix, except
	\begin{align}
	\Ext^1_{E(n)}\l(\pgfrac{4}{2},(i)\r)&=0, \quad i \not = 4,5 \\
	\Ext^1_{E(n)}\l(\pgfrac{4}{5},(i)\r)&=0, \quad i \not = 2,4.
	\end{align}
	
	Let us prove this. As a vector space,
	
	\[
	\pgfrac{4}{2}=\frac{\kk e_4 \oplus \kk x_2 \oplus \kk x_5 \oplus \kk a_4^1 \oplus \cdots \oplus \kk a_4^{N_4}}{\kk x_5 \oplus \kk a_4^1 \oplus \cdots \oplus \kk a_4^{N_4}}.
	\]
	
	We have the following projective resolution of $\pgfrac{4}{2}:$
	
	\[
	\begin{tikzcd}
	\cdots \arrow[r] & P_5 \oplus P_4^{N_4} \arrow[r] \arrow[rd] & P_4 \arrow[r] & \pgfrac{4}{2} \arrow[r] & 0 \\
	{} & {} & \kk x_5 \oplus \kk a_4^1 \oplus \cdots \oplus \kk a_4^{N_4} \arrow[hookrightarrow]{u} & {} & {} & {}
	\end{tikzcd}
	\]
	
	Take $\Hom_{E(n)}(-,(i))=\Hom_{E(n)}(-,S_i)$:
	
	\[
	\begin{tikzcd}
	0 \arrow[r] & \Hom_{E(n)}(P_4, S_i) \arrow[r] & \Hom_{E(n)}(P_5,S_i) \oplus \Hom_{E(n)}(P_4,S_i)^{N_4} \arrow[r] & \cdots
	\end{tikzcd}
	\]
	
	If $i\not= 4, 5$, then $\Hom_{E(n)}(P_5,S_i) \oplus \Hom_{E(n)}(P_4,S_i)^{N_4}=0,$ so $\Ext^1_{E(n)}\l(\pgfrac{4}{2},(i) \r)=0.$
	
	\vsp
	
	Similarly, as a vector space we have
	\[
	\pgfrac{4}{5}=\frac{\kk e_4 \oplus \kk x_2 \oplus \kk x_5 \oplus \kk a_4^1 \oplus \cdots \oplus \kk a_4^{N_4}}{\kk x_2 \oplus \kk a_4^1 \oplus \cdots \oplus \kk a_4^{N_4}}.
	\]
	
	We have the following projective resolution of $\pgfrac{4}{2}:$
	
	\[
	\begin{tikzcd}
	\cdots \arrow[r] & P_2 \oplus P_4^{N_4} \arrow[r] \arrow[rd] & P_4 \arrow[r] & \pgfrac{4}{5} \arrow[r] & 0 \\
	{} & {} & \kk x_2 \oplus \kk a_4^1 \oplus \cdots \oplus \kk a_4^{N_4} \arrow[hookrightarrow]{u} & {} & {} & {}
	\end{tikzcd}
	\]
	
	Take $\Hom_{E(n)}(-,(i))=\Hom_{E(n)}(-,S_i)$:
	
	\[
	\begin{tikzcd}
	0 \arrow[r] & \Hom_{E(n)}(P_4, S_i) \arrow[r] & \Hom_{E(n)}(P_2,S_i) \oplus \Hom_{E(n)}(P_4,S_i)^{N_4} \arrow[r] & \cdots
	\end{tikzcd}
	\]
	
	If $i\not= 2,4$, then $\Hom_{E(n)}(P_2,S_i) \oplus \Hom_{E(n)}(P_4,S_i)^{N_4}=0,$ so $\Ext^1_{E(n)}\l(\pgfrac{4}{5},(i) \r)=0.$
	
	\vsp
	
	{\bf Claim.}
	For any brick set $\phi,$ $\rho(A(\phi)) \leq \max\{N_1, \ldots, N_n\}$. Because $\phi = \{(1), \ldots, (n)\}$ is a brick set and $\rho(A(\phi))=\max\{N_1, \ldots, N_n\}$, we must have $\fpd \l(E(n)\dash\bmod\r)=\max\{N_1, \ldots, N_n\}.$
	
	\vsp
	
	Before we begin, note that by examination of the $\Hom$ and $\Ext^1$ matrices, if $(i) \in \l\{(1), \ldots, (n)\r\}$ and $\pgfrac{j}{k} \in \l\{ \pgfrac{1}{3}, \pgfrac{3}{4}, \pgfrac{5}{6}, \ldots, \pgfrac{n-1}{n}\r\}$ are in the same brick set, then 
	$$\Ext^1_{E(n)}\l((i),\pgfrac{j}{k}\r)=0=\Ext^1_{E(n)}\l(\pgfrac{j}{k},(i)\r).$$
	
	\vsp
	
	\emph{Case 1: $\pgfrac{4}{2},\pgfrac{4}{5}, \pgfrac{4}{2\:\:5} \not \in \phi$.}
	
	\vsp
	
	Then 
	$$\phi=\l\{(i_1), \ldots, (i_M), \pgfrac{j_1}{k_1}, \ldots, \pgfrac{j_L}{k_L}\r\},$$
	where 
	$$(i_m) \in \l\{(1), \ldots, (n)\r\}, \quad \pgfrac{j_l}{k_l} \in \l\{ \pgfrac{1}{3}, \pgfrac{3}{4}, \pgfrac{5}{6}, \ldots, \pgfrac{n-1}{n}\r\}.$$
	
	Then $A(\phi)$ is a principal submatrix of the following upper triangular matrix, whose eigenvalues are contained in the set $\{0, N_1, \ldots, N_n\}$. Therefore, $\rho(A(\phi)) \leq \max\{N_1, \ldots, N_n\}$.
	
	\[
  	\begin{blockarray}{cccccccc|ccc}
    	{} & (1) & (3) & (4) & (2) & (5) & \cdots & (n) & \pgfrac{j_1}{k_1} & \cdots & \pgfrac{j_L}{k_L}  \\
    	\begin{block}{c(ccccccc|ccc)}
    	(1) & N_1 & 1 & 0 & 0 & 0 & 0 & 0 & 0 & 0 & 0  \\
    	(3) & 0 & N_3 & 1 & 0 & 0 & 0 & 0 & 0 & 0 & 0  \\
	(4) & 0 & 0 & N_4 & 1 & 1 & 0 & 0 & 0 & 0 & 0  \\
	(2) & 0 & 0 & 0 & N_2 & 0 & 0 & 0 & 0 & 0 & 0  \\
    	(5) & 0 & 0 & 0 & 0 & N_5 & * & * & 0 & 0 & 0  \\
    	\vdots & 0 & 0 & 0 & 0 & 0 & \ddots & * & 0 & 0 & 0 \\
	(n) & 0 & 0 & 0 & 0 & 0 & 0 & N_n & 0 & 0 & 0 \\
    	\cline{1-11}
    	\pgfrac{j_1}{k_1} & 0 & 0 & 0 & 0 & 0 & 0 & 0 & 0 & 0 & 0    \\
    	\vdots & 0 & 0 & 0 & 0 & 0 & 0 & 0 & 0 & 0 & 0   \\
	\pgfrac{j_L}{k_L} & 0 & 0 & 0 & 0 & 0 & 0 & 0 & 0 & 0 & 0   \\
	\end{block}
  	\end{blockarray}
	\]
	
	\vsp
	
	\emph{Case 2: $\pgfrac{4}{2\:\:5} \in \phi.$}
	
	\vsp
	
	Then by examination of the $\Hom$ matrix, $(2),(4), \pgfrac{4}{2}, \pgfrac{4}{5} \not \in \phi.$ So 
	$$\phi=\l\{(i_1), \ldots, (i_M), \pgfrac{j_1}{k_1}, \ldots, \pgfrac{j_L}{k_L}, \pgfrac{4}{2\:\:5}\r\},$$
	where 
	$$(i_m) \in \l\{(1), \ldots, (n)\r\}, \quad \pgfrac{j_l}{k_l} \in \l\{ \pgfrac{1}{3}, \pgfrac{3}{4}, \pgfrac{5}{6}, \ldots, \pgfrac{n-1}{n}\r\}.$$
	
	Then $A(\phi)$ is given by the following upper triangular matrix, whose eigenvalues are contained in the set $\{0, N_1, \ldots, N_n\}$. Notice that it is upper triangular because $(4) \not \in \phi$. Therefore, $\rho(A(\phi)) \leq \max\{N_1, \ldots, N_n\}$.
	
	\[
  	\begin{blockarray}{cccc|ccc|c}
    	{} & (i_1) & \cdots & (i_M) & \l(\gfrac{j_1}{k_1}\r) & \cdots & \pgfrac{j_L}{k_L} & \pgfrac{4}{2\:\:5} \\
    	\begin{block}{c(ccc|ccc|c)}
    	(i_1) & N_{i_1} & * & * & 0 & 0 & 0  & *  \\
	\vdots & 0 & \ddots & * & 0 & 0 & 0  & * \\
	(i_M) & 0 & 0 & N_{i_M} & 0 & 0 & 0  & * \\
    	\cline{1-8}
    	\pgfrac{j_1}{k_1} & 0 & 0 & 0 & 0 & 0 & 0 & 0    \\
    	\vdots & 0 & 0 & 0 & 0 & 0 & 0 & 0    \\
	\pgfrac{j_L}{k_L} & 0 & 0 & 0 & 0 & 0 & 0 & 0    \\
	\cline{1-8}
	\pgfrac{4}{2\:\:5} & 0 & 0 & 0 & 0 & 0 & 0 & 0  \\
	\end{block}
  	\end{blockarray}
	\]
	
	\vsp
	
	\emph{Case 3: $\pgfrac{4}{2}, \pgfrac{4}{5} \in \phi.$}
	
	\vsp
	
	Then $(2),(4),(5), \pgfrac{4}{2\:\:5} \not \in \phi$. So
	$$\phi=\l\{(i_1), \ldots, (i_M), \pgfrac{j_1}{k_1}, \ldots, \pgfrac{j_L}{k_L}, \pgfrac{4}{2},\pgfrac{4}{5}\r\},$$
	where 
	$$(i_m) \in \l\{(1), \ldots, (n)\r\}, \quad \pgfrac{j_l}{k_l} \in \l\{ \pgfrac{1}{3}, \pgfrac{3}{4}, \pgfrac{5}{6}, \ldots, \pgfrac{n-1}{n}\r\}.$$
	
	Then $A(\phi)$ is given by the following upper triangular matrix, whose eigenvalues are contained in the set $\{0, N_1, \ldots, N_n\}$. Notice that it is upper triangular because $(4) \not \in \phi$. Therefore, $\rho(A(\phi)) \leq \max\{N_1, \ldots, N_n\}$.
	
	\[
  	\begin{blockarray}{cccc|ccc|cc}
    	{} & (i_1) & \cdots & (i_M) & \l(\gfrac{j_1}{k_1}\r) & \cdots & \pgfrac{j_L}{k_L} & \pgfrac{4}{2} & \pgfrac{4}{5} \\
    	\begin{block}{c(ccc|ccc|cc)}
    	(i_1) & N_{i_1} & * & * & 0 & 0 & 0 & 0 & 0  \\
	\vdots & 0 & \ddots & * & 0 & 0 & 0 & 0 & 0 \\
	(i_M) & 0 & 0 & N_{i_M} & 0 & 0 & 0 & 0  & 0 \\
    	\cline{1-9}
    	\pgfrac{j_1}{k_1} & 0 & 0 & 0 & 0 & 0 & 0 & 0 & 0    \\
    	\vdots & 0 & 0 & 0 & 0 & 0 & 0 & 0  & 0  \\
	\pgfrac{j_L}{k_L} & 0 & 0 & 0 & 0 & 0 & 0 & 0 & 0   \\
	\cline{1-9}
	\pgfrac{4}{2} & 0 & 0 & 0 & 0 & 0 & 0 & 0 & 0  \\
	\pgfrac{4}{5} & 0 & 0 & 0 & 0 & 0 & 0 & 0 & 0  \\
	\end{block}
  	\end{blockarray}
	\]
	
	\vsp
	
	\emph{Case 4: $\pgfrac{4}{2}\in \phi, \pgfrac{4}{5} \not \in \phi$.}	
	
	\vsp
	
	Then $(2),(4), \pgfrac{4}{2\:\:5} \not \in \phi$. So
	$$\phi=\l\{\pgfrac{4}{2}, (i_1), \ldots, (i_M), \pgfrac{j_1}{k_1}, \ldots, \pgfrac{j_L}{k_L}\r\},$$
	where 
	$$(i_m) \in \l\{(1), \ldots, (n)\r\}, \quad \pgfrac{j_l}{k_l} \in \l\{ \pgfrac{1}{3}, \pgfrac{3}{4}, \pgfrac{5}{6}, \ldots, \pgfrac{n-1}{n}\r\}.$$
		
	Then $A(\phi)$ is given by the following upper triangular matrix, whose eigenvalues are contained in the set $\{0, N_1, \ldots, N_n\}$. Notice that it is upper triangular because $(4) \not \in \phi$. Therefore, $\rho(A(\phi)) \leq \max\{N_1, \ldots, N_n\}$.
	
	\[
  	\begin{blockarray}{cc|ccc|ccc}
    	{} & \pgfrac{4}{2} & (i_1) & \cdots & (i_M) & \l(\gfrac{j_1}{k_1}\r) & \cdots & \pgfrac{j_L}{k_L}  \\
    	\begin{block}{c(c|ccc|ccc)}
	\pgfrac{4}{2} & 0 & * & * & * & 0 & 0 & 0  \\
	\cline{1-8}
    	(i_1) & 0 & N_{i_1} & * & * & 0 & 0 & 0   \\
	\vdots & 0 & 0 & \ddots & * & 0 & 0 & 0  \\
	(i_M) & 0 & 0 & 0 & N_{i_M} & 0 & 0 & 0  \\
    	\cline{1-8}
    	\pgfrac{j_1}{k_1} & 0 & 0 & 0 & 0 & 0 & 0 & 0    \\
    	\vdots & 0 & 0 & 0 & 0 & 0 & 0 & 0    \\
	\pgfrac{j_L}{k_L} & 0 & 0 & 0 & 0 & 0 & 0 & 0    \\
	\end{block}
  	\end{blockarray}
	\]

	\vsp
	
	\emph{Case 5: $\pgfrac{4}{2}\not \in \phi, \pgfrac{4}{5} \in \phi$.}	
	
	\vsp
	
	Then $(4),(5), \pgfrac{4}{2\:\:5} \not \in \phi$. So
	$$\phi=\l\{\pgfrac{4}{5}, (i_1), \ldots, (i_M), \pgfrac{j_1}{k_1}, \ldots, \pgfrac{j_L}{k_L}\r\},$$
	where 
	$$(i_m) \in \l\{(1), \ldots, (n)\r\}, \quad \pgfrac{j_l}{k_l} \in \l\{ \pgfrac{1}{3}, \pgfrac{3}{4}, \pgfrac{5}{6}, \ldots, \pgfrac{n-1}{n}\r\}.$$
	
	Then $A(\phi)$ is given by the following upper triangular matrix, whose eigenvalues are contained in the set $\{0, N_1, \ldots, N_n\}$. Notice that it is upper triangular because $(4) \not \in \phi$. Therefore, $\rho(A(\phi)) \leq \max\{N_1, \ldots, N_n\}$.
	
	\[
  	\begin{blockarray}{cc|ccc|ccc}
    	{} & \pgfrac{4}{5} & (i_1) & \cdots & (i_M) & \l(\gfrac{j_1}{k_1}\r) & \cdots & \pgfrac{j_L}{k_L}  \\
    	\begin{block}{c(c|ccc|ccc)}
	\pgfrac{4}{5} & 0 & * & * & * & 0 & 0 & 0  \\
	\cline{1-8}
    	(i_1) & 0 & N_{i_1} & * & * & 0 & 0 & 0   \\
	\vdots & 0 & 0 & \ddots & * & 0 & 0 & 0  \\
	(i_M) & 0 & 0 & 0 & N_{i_M} & 0 & 0 & 0  \\
    	\cline{1-8}
    	\pgfrac{j_1}{k_1} & 0 & 0 & 0 & 0 & 0 & 0 & 0    \\
    	\vdots & 0 & 0 & 0 & 0 & 0 & 0 & 0    \\
	\pgfrac{j_L}{k_L} & 0 & 0 & 0 & 0 & 0 & 0 & 0    \\
	\end{block}
  	\end{blockarray}
	\]
	
	\vsp
	
	We have exhausted all possible cases. Therefore for $n\in\{6,7,8\}$, 
	\[
	\fpd \l(E(n)\dash\bmod \r)= \max\{N_1, \ldots, N_n\}.
	\]
	\end{proof}
	
	\vsp


\section{Modified ADE quiver algebras with different arrow orientations}

	We have calculated the $\fpd$ of modified ADE quiver algebras with arrows in a certain direction (\cref{thm:ADE}). It is natural to ask what happens if the directions of the arrows change. It is clear that the $\fpd$ remains the same upon taking the opposite algebra, which is the radical square zero bound quiver algebra of the opposite quiver \cite{CGWZZZ}. However, we do not know what happens in other cases.
	
	It is possible that the directions of the arrows do not matter (\cref{conj:ADE}). We will show that the directions of the arrows do not matter for $A(n)$ if $n \leq 3.$ Preliminary calculations for $n=4$ also suggest a positive answer to this question, but the question requires further analysis.
	
\begin{theorem} \label{thm:ADEtilde3}
For the modified ADE bound quiver algebras defined in \cref{def:modifiedADE}, if $A \in \{A(1),A(2),A(3)\}$, then $\fpd \l(A \dash \bmod\r)$ is invariant under change of direction of the arrows $x_i$.
\end{theorem}

\begin{proof} 

	The proof for $n=1, 2$ is clear because there is only one possible direction for the arrows up to isomorphism.
	
	For $n=3,$ let $A$ be the finite-dimensional algebra described by the following quiver, with relations such that paths of length greater than or equal to 2 are 0:
	
	\[
	\begin{tikzcd}
		1  \arrow[out=155, in=205, loop] \arrow[out=162, in=198, loop]  \arrow[out=148, in=212, loop, swap, "a_n"] \arrow[r,"x"]  & 2  \arrow[out=238, in=302, loop, swap, "b_m"] \arrow[out=250, in=290, loop] & 3 \arrow[l, swap, "y"] \arrow[out=32, in=328, loop, "c_l"]		
	\end{tikzcd} 
	\]
	
	Label the loops $a_1, \ldots, a_N, b_1, \ldots, b_M, c_1, \ldots, c_L.$ The brick objects of $A$ are in one-to-one correspondence with the brick objects in $B$, where $B$ is the finite dimensional algebra described by the following quiver:
	
	\[
	\begin{tikzcd}
		1 \arrow[r, "x"] & 2 & 3 \arrow[l, swap, "y"]
	\end{tikzcd}
	\]
	
	The six non-isomorphic indecomposable representations of $B$ are well known \cite[Ex. 1.14]{Sch}. We have the three simples $S_1=I_1, S_2=P_2, S_3=I_3,$ and the additional representations are $P_1, P_3, I_2.$ All of these are brick objects.
	
	\begin{align*}
	S_1&=(1)=(1,0,0) \\
	S_2&=(2)=(0,1,0) \\
	S_3&=(3)=(0,0,1) \\
	P_1&=\pgfrac{1}{2}=(1,1,0) \\
	P_3&=\pgfrac{3}{2}=(0,1,1) \\
	I_2&=\pgfrac{1 \:\: 3}{2}=(1,1,1)
	\end{align*}
	
	We can determine the brick sets by examining the $\Hom$ matrix below, where $H_{ij}=\dim \Hom(i,j).$ The $*$ represents a non-zero entry. 
	\[
	H = 
    	\bordermatrix{ & S_1 & S_2 & S_3 & P_1 & I_2 & P_3 \cr
      	S_1 & 1 & 0 & 0 & 0 & 0 & 0 \cr
      	S_2 & 0 & 1 & 0 & * & 1 & * \cr
      	S_3 & 0 & 0 & 1 & 0 & 0 & 0 \cr
	P_1 & * & 0 & 0 & 1 & * & 0 \cr
	I_2 & * & 0 & * & 0 & 1 & 0 \cr
	P_3 & 0 & 0 & * & 0 & * & 1}
	\]

	\vsp
	
	The maximal brick sets are:
	\begin{align*}
		\phi_1&=\{I_2\}, \\
		\phi_2&=\{S_1, S_2, S_3\}, \\
		\phi_3&= \{ P_3, S_1 \}, \\
		\phi_4&= \{P_1, P_3\}, \\
		\phi_5&=\{P_1, S_3\}.
	\end{align*}
	
	We have $\rho(A(\phi_1))=|\Ext^1_A(I_2,I_2)|=0$ by \cref{lem:rad} \cref{lem:radinj} since $I_2 \not \cong S_2$.
	
	Since $|\Ext^1_A(S_i, S_j)|$ is the number of arrows from $i$ to $j$, we have
	$$\rho(A(\phi_2))=\rho 	\begin{pmatrix} 
						N & 1 & 0 \\ 
						0 & M & 0 \\
						0 & 1 & L		
						\end{pmatrix} 
				=\max \{N,M,L\}.$$
				
	We have by \cref{lem:rad} \cref{lem:radproj}, \cref{lem:simproj} that
	\begin{align*}
	\rho(A(\phi_3))	&=\rho 	\begin{pmatrix}
						|\Ext^1(P_3,P_3)| & |\Ext^1(P_3, S_1)| \\
						|\Ext^1(S_1, P_3)| & |\Ext^1(S_1,S_1)|
						\end{pmatrix} \\
				&= \rho	\begin{pmatrix}
						0 & 0 \\
						|\Ext^1(S_1, P_3)| & N
						\end{pmatrix} \\
				&=N.
	\end{align*}
	
	By \cref{lem:rad} we have
	\begin{align*}
	\rho(A(\phi_4))	 	&=\rho	\begin{pmatrix}
							0 & 0 \\
							0 & 0
							\end{pmatrix} =0.
	\end{align*}
	
	By \cref{lem:rad} \cref{lem:radproj}, \cref{lem:simproj} we have
	\begin{align*}
	\rho(A(\phi_5))	&=\rho 	\begin{pmatrix}
						|\Ext^1(P_1,P_1)| & |\Ext^1(P_1, S_3)| \\
						|\Ext^1(S_3, P_1)| & |\Ext^1(S_3,S_3)|
						\end{pmatrix} \\
				&= \rho	\begin{pmatrix}
						0 & 0 \\
						|\Ext^1(S_3, P_1)| & L
						\end{pmatrix} \\
				&=L.
	\end{align*}
	
	Therefore, 
	$$\fpd \l(A\dash\bmod\r)=\max\{N,M,L\}.$$
	
	Every modified ADE bound quiver algebra with arrows in an arbitrary orientation is isomorphic to one of type $A, A^{op},$ or $A(3)$. Since $\fpd$ is invariant under taking the opposite algebra by \cite[Cor. 3.10]{CGWZZZ}, the theorem is true for $n=3$.
	\end{proof}
	
	\vsp


\bigskip

\subsection*{Acknowledgments}
The author would like to thank her advisor James Zhang
for many useful conversations on the subject. This work was partially 
supported by the US National Science Foundation (Grant Nos. 
DMS-1402863 and DMS-1700825). 

\bibliographystyle{alpha}

\bibliography{bibliography}

\end{document}